\documentclass[10pt,a4paper,twoside]{amsart}

\usepackage{graphicx}
\usepackage{a4wide}

\newtheorem{definition}{Definition}[section]
\newtheorem{theorem}{Theorem}[section]
\newtheorem{lemma}{Lemma}[section]

\newtheorem*{maintheorem*}{Main Theorem}
\allowdisplaybreaks
\numberwithin{equation}{section}

\newcommand{\Kruzkov}{Kru\v{z}kov~}

\newcommand{\Set}[1]{\left\{#1\right\}}

\newcommand{\norm}[1]{\left\| #1 \right\|}

\newcommand{\eps}{\varepsilon}

\newcommand{\ed}{{\eps,\delta}}
\newcommand{\ued}{u_\ed}
\newcommand{\Ped}{P_\ed}

\newcommand{\ue}{u_\eps}

\newcommand{\uek}{u_{\eps_k}}

\newcommand{\Pe}{P_\eps}

\newcommand{\Pek}{P_{\eps_k}}

\newcommand{\Fe}{F_\eps}

\newcommand{\pt}{\partial_t}

\newcommand{\px}{\partial_x }
\newcommand{\pxx}{\partial_{xx}^2}

\newcommand{\ptx}{\partial_{tx}^2}

\renewcommand{\i}{\ifmmode\mathit{\mathchar"7010 }\else\char"10 \fi}
\renewcommand{\j}{\ifmmode\mathit{\mathchar"7011 }\else\char"11 \fi}
\newcommand{\R}{\mathbb{R}}
\newcommand{\N}{\mathbb{N}}

\newcommand{\Hneg}{H_{\mathrm{loc}}^{-1}}
\newcommand{\CL}{\mathcal{L}}
\newcommand{\CLea}{\mathcal{L}}

\newcommand{\sgn}[1]{\mathrm{sign}\left(#1\right)}

{%

\begin{enumerate}}%
{\end{enumerate}}

{%

\begin{enumerate}}%
{\end{enumerate}}

\begin{document}\large

\title[The short pulse equation]{Wellposedness results \\ for the Short Pulse Equation}
\author[G. M. Coclite and L. di Ruvo]{Giuseppe Maria Coclite and Lorenzo di Ruvo}
\address[Giuseppe Maria Coclite and Lorenzo di Ruvo]
{\newline Department of Mathematics,   University of Bari, via E. Orabona 4, 70125 Bari,   Italy}
\email[]{giuseppemaria.coclite@uniba.it, lorenzo.diruvo@uniba.it}
\urladdr{http://www.dm.uniba.it/Members/coclitegm/}
\date{\today}



\begin{abstract}
The short pulse equation provides a model for the propagation of ultra-short light pulses in silica optical fibers. It is a nonlinear evolution equation.
In this paper  the wellposedness of bounded solutions for the homogeneous initial boundary value problem and the Cauchy problem associated to this equation are studied.
\end{abstract}

\maketitle


\section{Introduction}
The short pulse equation which has the form
\begin{equation}
\label{eq:SPE}
\px\left(\pt u -\frac{1}{6}\px u^3\right)= \gamma u, \quad \gamma>0,
\end{equation}
up to a scale transformation of its variables, was introduced recently by
Sch\"afer and Wayne \cite{SW} as  a model equation describing the propagation of ultra-short light pulses in silica optical fibers. 
It provides also an approximation of nonlinear wave packets in dispersive media in the limit of few cycles on the ultra-short pulse scale.
Numerical simulations \cite{CJSW} show that the short pulse equation approximation 
to Maxwell's equations in the case when the pulse spectrum is not narrowly localized 
around the carrier frequency is better than the one obtained from the nonlinear Schr\"odinger 
equation, which models the evolution of slowly varying wave trains.
Such ultra-short plays a key role in the development of future technologies of ultra-fast optical transmission of informations.

In \cite{B} the author studied a new hierarchy of equations containing the short pulse equation \eqref{eq:SPE} and the elastic 
beam equation, which describes nonlinear transverse oscillations of elastic beams 
under tension. He showed that the hierarchy of equations is integrable. He obtained the two compatible Hamiltonian structures
and constructs an infinite series of both local and nonlocal conserved charges. Moreover, he gave the Lax description for both systems. 
The integrability and the existence of solitary wave solutions have been studied in \cite{SS, SS1}.

Well-posedness and wave breaking for the short pulse equation have been 
studied in \cite{SW} and \cite{LPS}, respectively. 
Our aim is to investigate the  well-posedness in classes of discontinuous functions for \eqref{eq:SPE}.
We consider both the initial boundary value problem  (see Section \ref{sec:1}) and the Cauchy problem (see Section \ref{sec:2}) for \eqref{eq:SPE}.

Integrating \eqref{eq:SPE} in $x$ we gain the integro-differential formulation of  \eqref{eq:SPE} (see \cite{SS})
\begin{equation*}
\pt u -\frac{1}{6} \px u^3=\gamma \int^x u(t,y) dy,
\end{equation*}
that is equivalent to
\begin{equation}
\label{eq:integ}
\pt u-\frac{1}{6} \px u^3=\gamma P,\qquad \px P=u.
\end{equation}

One of the main issues in the analysis of \eqref{eq:integ} is that the equation is not preserving the $L^1$ norm, the unique useful conserved quantities are 
\begin{equation*}
t\longmapsto\int u(t,x)dx,\qquad t\longmapsto\int u^2(t,x)dx.
\end{equation*}
As a consequence the nonlocal source term $P$ and the solution $u$ are a priori only locally bounded. 
Since we are interested in the bounded solutions of \eqref{eq:SPE}, some assumptions on the decay at infinity of the initial condition $u_0$ is needed.
 Regarding the flux function, here we use the cubic one 
\begin{equation*}
u\longmapsto- \frac{u^3}{6}
\end{equation*}
because this is the one that appears in the original short-pulse equation. Anyway all our arguments can be generalized to subcubic genuinely nonlinear fluxes.
The genuine nonlinearity assumption is necessary for the compactness argument based on the compensated compactness. 
The subcubic assumption together with the assumptions on the on the decay at infinity of the initial condition $u_0$ guarantees the boundedness of the solutions.


Our existence argument is based on passing to the limit using a compensated compactness argument \cite{TartarI} in a vanishing viscosity approximation of \eqref{eq:SPE}:
\begin{equation*}
\pt \ue-\frac{1}{6}\px \ue^3=\gamma\Pe+ \eps\pxx\ue,\qquad -\eps\pxx\Pe+\px\Pe=\ue.
\end{equation*}
On the other hand we use the \Kruzkov doubling of variables  method  \cite{K} for the uniqueness and stability of the solutions of \eqref{eq:SPE}.

\section{The initial boundary value problem}\label{sec:1}
In this section, we augment \eqref{eq:SPE} with the boundary condition
\begin{equation}
\label{eq:boundary}
u(t,0)=0, \qquad t>0,
\end{equation}
and the initial datum
\begin{equation}
\label{eq:init}
u(0,x)=u_0(x), \qquad x>0.
\end{equation}
We assume that
\begin{equation}
\label{eq:assinit}
u_0\in L^{\infty}(0,\infty)\cap L^{1}(0,\infty), \quad \int_{0}^{\infty}u_{0}(x) dx =0.
\end{equation}
On the function
\begin{equation}
\label{eq:def-di-P0}
P_{0}(x)=\int_{0}^{x}u_{0}(y)dy,
\end{equation}
we assume that
\begin{equation}
\label{eq:l-2-di-P0}
\norm{P_{0}}^2_{L^2(0,\infty)}=\int_{0}^{\infty}\left(\int_{0}^{x} u_{0}(y)dy\right)^2 dx < \infty.
\end{equation}

Integrating \eqref{eq:SPE} on $(0,x)$ we obtain the integro-differential formulation of the initial-boundary value problem \eqref{eq:SPE}, \eqref{eq:boundary},
  \eqref{eq:init} (see \cite{SS})
\begin{equation}
\label{eq:SPE-u}
\begin{cases}
\displaystyle \pt u-\frac{1}{6}\px u^3=\gamma \int^x_0 u(t,y) dy,&\qquad t>0, \ x>0,\\
u(t,0)=0,& \qquad t>0,\\
u(0,x)=u_0(x), &\qquad x>0.
\end{cases}
\end{equation}
This is equivalent to
\begin{equation}
\label{eq:SPEw}
\begin{cases}
\displaystyle\pt u-\frac{1}{6}\px u^3=\gamma P,&\qquad t>0, \ x>0 ,\\
\px P=u,&\qquad t>0, \ x>0,\\
u(t,0)=0, &\qquad t>0,\\
P(t,0)=0,& \qquad t>0,\\
u(0,x)=u_0(x), &\qquad x>0.
\end{cases}
\end{equation}
Due to the regularizing effect of the P equation in \eqref{eq:SPEw} we have that
\begin{equation}
\label{eq:SPEsmooth}
    u\in L^{\infty}((0,T)\times(0,\infty))\Longrightarrow P\in L^{\infty}((0,T);W^{1_,\infty}(0,\infty)), \quad T>0.
\end{equation}
Therefore, if a map $u\in  L^{\infty}((0,T)\times(0,\infty)),\,T>0,$  satisfies, for every convex
map $\eta\in  C^2(\R)$,
\begin{equation}
\label{eq:SPEentropy}
   \pt \eta(u)+ \px q(u)-\gamma\eta'(u) P\le 0, \qquad     q(u)=-\int^u \frac{\xi^2}{2} \eta'(\xi)\, d\xi,
\end{equation}
in the sense of distributions, then \cite[Theorem 1.1]{CKK} provides the existence of strong trace $u^\tau_0$ on the
boundary $x=0$.

\begin{definition}
\label{def:sol}
We say that  $u\in  L^{\infty}((0,T)\times(0,\infty))$, $T>0,$ is an entropy solution of the initial-boundary
value problem \eqref{eq:SPE}, \eqref{eq:boundary}, and  \eqref{eq:init} if
\begin{itemize}
\item[$i$)] $u$ is a distributional solution of \eqref{eq:SPE-u} or equivalently of \eqref{eq:SPEw};
\item[$ii$)] for every convex function $\eta\in  C^2(\R)$ the
entropy inequality \eqref{eq:SPEentropy} holds in the sense of distributions in $(0,\infty)\times(0,\infty)$;
\item[$iii$)] for every convex function $\eta\in  C^2(\R)$ with corresponding $q$ defined by $q'(u)=-\frac{u^2}{2}\eta'(u)$,
the boundary entropy condition
\begin{equation}
\label{eq:SPEentropyboundary}
\begin{split}
q(u^\tau_0(t))-q(0)-\eta'(0)\frac{(u_0^\tau(t))^3}{6}\le 0
\end{split}
\end{equation}
holds for a.e.~$t\in(0,\infty)$, where $u_0^\tau(t)$ is the trace of $u$ on the
boundary $x=0$.
\end{itemize}
\end{definition}
We observe that the previous definition is equivalent to the following family of inequalities
inequality \\(see \cite{BRN}):
\begin{equation}
\label{eq:ent1}
\begin{split}
\int_{0}^{\infty}\!\!\!\!\int_{0}^{\infty}(\vert u - c\vert\pt\phi&-\sgn{u-c}\left(\frac{u^3}{6}-\frac{c^3}{6}\right)\px\phi)dtdx\\
&+\gamma\int_{0}^{\infty}\!\!\!\!\int_{0}^{\infty}\sgn{u-c}Pdtdx\\
&+\int_{0}^{\infty}\sgn{c}\left(\frac{((u^\tau_0(t)))^3}{6}-\frac{c^3}{6}\right)dt\\
&+\int_{0}^{\infty}\vert u_{0}(x)-c\vert\phi(0,x)dx\geq 0,
\end{split}
\end{equation}
for every non-negative test function  $\phi\in C^{\infty}(\R^2)$ with
compact support, and for every $c\in \R$.

The main result of this section is the following theorem.
\begin{theorem}
\label{th:main}
Assume  \eqref{eq:assinit} and \eqref{eq:l-2-di-P0}.
The initial-boundary value problem
\eqref{eq:SPE}, \eqref{eq:boundary} and \eqref{eq:init} possesses
an unique entropy solution $u$ in the sense of Definition \ref{def:sol}. In particular, we have that
\begin{equation}
\label{eq:u-media-nulla}
\int_{0}^{\infty}u(t,x)dx =0, \quad t>0.
\end{equation}
Moreover, if $u$ and $v$ are two entropy solutions \eqref{eq:SPE}, \eqref{eq:boundary}, \eqref{eq:init} in the sense of Definition \ref{def:sol}, the following inequality holds
 \begin{equation}
 \label{eq:stability}
\norm{u(t,\cdot)-v(t,\cdot)}_{L^1(0,R)}\le  e^{C(T) t}\norm{u(0,\cdot)-v(0,\cdot)}_{L^1(0,R+C(T)t)},
\end{equation}
for almost every $0<t<T$, $R>0$, and some suitable constant $C(T)>0$.
\end{theorem}

A similar result has been proved in \cite{CdK, dR} in the context of locally bounded solutions.

Our existence argument is based on passing to the limit in a vanishing viscosity approximation of \eqref{eq:SPE}.

Fix a small number $0<\eps<1$, and let $\ue=\ue (t,x)$ be the unique classical solution of the following mixed problem \cite{CHK:ParEll}
\begin{equation}
\label{eq:SPEsw}
\begin{cases}
\displaystyle\pt \ue -\frac{1}{2}\ue^2\px\ue=\gamma\Pe+ \eps\pxx\ue,&\quad t>0,\ x>0,\\
-\eps\pxx\Pe+\px\Pe=\ue,&\quad t>0,\ x>0,\\
\ue(t,0)=0, &\quad t>0,\\
\Pe(t,0)=0,&\quad t>0,\\
\ue(0,x)=u_{\eps,0}(x),&\quad x>0,
\end{cases}
\end{equation}
where $u_{\eps,0}$ is a $C^\infty$ approximation of $u_{0}$ such that
\begin{equation}
\label{eq:u0eps}
\begin{split}
&\norm{u_{\eps,0}}_{L^2(0,\infty)}\le \norm{u_0}_{L^2(0,\infty)}, \quad \norm{u_{\eps,0}}_{L^{\infty}(0,\infty)}\le \norm{u_0}_{L^{\infty}(0,\infty)},\\
&\norm{P_{\eps,0}}_{L^2(0,\infty)}\le \norm{P_{0}}_{L^2(0,\infty)},\quad\eps\norm{\px P_{\eps,0}}_{L^2(0,\infty)}\le C_{0},
\end{split}
\end{equation}
and $C_0$ is a constant independent on $\eps$.

Let us prove some a priori estimates on $\ue$ and $\Pe$, denoting with $C_0$ the constants which depend on the initial datum, and $C(T)$ the constants which depend also on $T$.

Arguing as \cite{Cd}, we obtain the following results

\begin{lemma}
\label{lm:cns}
For each $t\in (0,\infty)$,
\begin{equation}
\label{eq:P-pxP-intfy}
\Pe(t,\infty)=\px\Pe(t,\infty)=0.
\end{equation}
Moreover,
\begin{equation}
\begin{split}
\label{eq:equ-L2-stima}
\eps^2\norm{\pxx\Pe(t,\cdot)}^2_{L^2(0,\infty)}&+\eps(\px\Pe(t,0))^2\\
&+ \norm{\px\Pe(t,\cdot)}^2_{L^2(0,\infty)}=\norm{\ue(t,\cdot)}^2_{L^2(0,\infty)}.
\end{split}
\end{equation}
\end{lemma}

\begin{lemma}
\label{lm:2}
For each $t\in(0,\infty)$,
\begin{align}
\label{eq:int-u}
\int_{0}^{\infty}\ue(t,x) dx &= \eps\px\Pe(t,0),\\
\label{eq:L-infty-Px}
\sqrt{\eps}\norm{\px\Pe(t, \cdot)}_{L^{\infty}(0,\infty)}&\le \norm{u(t,\cdot)}_{L^2(0,\infty)},\\
\label{eq:uP}
\int_{0}^{\infty}\ue(t,x)\Pe(t,x) dx&\le \norm{u(t,\cdot)}^2_{L^2(0,\infty)}.
\end{align}
\end{lemma}

\begin{lemma}\label{lm:stima-l-2}
For each $t\in(0,\infty)$, the inequality holds
\begin{equation}
\label{eq:stima-l-2}
\norm{\ue(t,\cdot)}^2_{L^2(0,\infty)}+ 2\eps e^{2\gamma t}\int_{0}^{\infty}e^{-2\gamma s}\norm{\px\ue(s,\cdot)}^2_{L^2(0,\infty)}ds\le e^{2\gamma t}\norm{u_{0}}^2_{L^2(0,\infty)}.
\end{equation}
In particular, we have
\begin{equation}
\label{eq:10021}
\begin{split}
\eps\norm{\pxx\Pe(t,\cdot)}_{L^2(0,\infty)}, \norm{\px\Pe(t,\cdot)}_{L^2(0,\infty)}&\le e^{\gamma t}\norm{u_{0}}_{L^2(0,\infty)},\\
\sqrt{\eps}\vert\px\Pe(t,0)\vert, \sqrt{\eps}\norm{\px\Pe(t, \cdot)}_{L^{\infty}(0,\infty)}&\le e^{\gamma t}\norm{u_{0}}_{L^2(0,\infty)}.
\end{split}
\end{equation}
Moreover, we get
\begin{equation}
\label{eq:258}
\norm{\Pe(t,\cdot)}_{L^{\infty}(0,\infty)}\le \sqrt{2e^{\gamma t}\norm{u_{0}}_{L^2(0,\infty)}\norm{\Pe(t,\cdot)}_{L^2(0,\infty)}}.
\end{equation}
\end{lemma}
\begin{proof}
Due to \eqref{eq:SPEsw} and \eqref{eq:uP},
\begin{align*}
\frac{d}{dt}\int_{0}^{\infty}\ue^2 dx =& 2 \int_{0}^{\infty}\ue\pt\ue dx\\
=& 2\eps \int_{0}^{\infty} \ue\pxx\ue dx + \int_{0}^{\infty}\ue^3 \px\ue dx + 2\gamma \int_{0}^{\infty} \ue\Pe dx\\
\le & - 2\eps\int_{0}^{\infty}(\px\ue)^2 dx  + 2\gamma \norm{\ue(t,\cdot)}^2_{L^2(0,\infty)}.
\end{align*}
The Gronwall Lemma and \eqref{eq:u0eps} give \eqref{eq:stima-l-2}.

\eqref{eq:10021} follows from \eqref{eq:equ-L2-stima}, \eqref{eq:L-infty-Px} and  \eqref{eq:stima-l-2}.

Finally, we prove \eqref{eq:258}.
Due to \eqref{eq:SPEsw} and H\"older inequality,
\begin{align*}
\Pe^2(t,x)\le 2\int_{0}^{\infty}\vert\Pe(t,x)\vert\vert\px\Pe(t,x)\vert dx\le  2\norm{\Pe(t,\cdot)}_{L^2(0,\infty)}\norm{\px\Pe(t,\cdot)}_{L^2(0,\infty)},
\end{align*}
that is
\begin{equation}
\label{eq:2592}
\vert\Pe(t,x)\vert\le \sqrt{2\norm{\Pe(t,\cdot)}_{L^2(0,\infty)}\norm{\px\Pe(t,\cdot)}_{L^2(0,\infty)}}.
\end{equation}
\eqref{eq:10021} and \eqref{eq:2592} give \eqref{eq:258}.
\end{proof}

\begin{lemma}
\label{lm:linfty-u}
For every $t\in(0,\infty)$,
\begin{equation}
\label{eq:linfty-u}
\norm{\ue(t,\cdot)}_{L^\infty(0,\infty)}\le\norm{u_0}_{L^\infty(0,\infty)}+\gamma\int_{0}^{t}\norm{\Pe(s,\cdot)}_{L^{\infty}(0,\infty)}ds.
\end{equation}
\end{lemma}

\begin{proof}
Due to \eqref{eq:SPEsw},
\begin{equation*}
\pt \ue -\frac{1}{2}\ue^2\px\ue-\eps\pxx \ue\le \gamma \vert\Pe(t,x)\vert \le \gamma \norm{\Pe(t,\cdot)}_{L^{\infty}(0,\infty)}.
\end{equation*}
Since the map
\begin{equation*}
{\mathcal F}(t):=\norm{u_0}_{L^\infty(0,\infty)}+\gamma\int_{0}^{t} \norm{\Pe(s,\cdot)}_{L^{\infty}(0,\infty)} ds,\qquad t\in(0,\infty),
\end{equation*}
solves the equation
\begin{equation*}
\frac{d{\mathcal F}}{dt}=\gamma \norm{\Pe(t,\cdot)}_{L^{\infty}(0,\infty)}
\end{equation*}
and
\begin{equation*}
\max\{\ue(0,x),0\}\le {\mathcal F}(t),\qquad (t,x)\in (0,\infty)^2,
\end{equation*}
the comparison principle for parabolic equations implies that
\begin{equation*}
 \ue(t,x)\le {\mathcal F}(t),\qquad (t,x)\in (0,\infty)^2.
\end{equation*}

In a similar way we can prove that
\begin{equation*}
\ue(t,x)\ge -{\mathcal F}(t),\qquad (t,x)\in (0,\infty)^2.
\end{equation*}
Therefore,
\begin{equation}
\label{eq:ass-val-u}
\vert\ue(t,x)\vert\le\norm{u_0}_{L^\infty(0,\infty)}+\gamma\int_{0}^{t} \norm{\Pe(s,\cdot)}_{L^{\infty}(0,\infty)} ds,
\end{equation}
which gives \eqref{eq:linfty-u}.
\end{proof}

\begin{lemma}\label{lm:def-di-f}
Consider the following function
\begin{equation}
\label{eq:def-di-f}
\Fe(t,x)=\int_{0}^{x} \Pe(t,y)dy, \quad t,\,x>0.
\end{equation}
We have that
\begin{equation}
\label{eq:lim-di-f}
\lim_{x\to\infty}\Fe(t,x)=\int_{0}^{\infty}\Pe(t,y)dy=\frac{\eps}{\gamma}\ptx\Pe(t,0)+\frac{\eps}{\gamma}\px\ue(t,0).
\end{equation}
\end{lemma}
\begin{proof}
Integrating on $(0,x)$ the first equation of \eqref{eq:SPEsw}, we get
\begin{equation}
\label{eq:int-in-x}
\int_{0}^{x}\pt\ue(t,y) dy -\frac{1}{6} \ue^3(t,x)  - \eps\px\ue(t,x) + \eps\px\ue(t,0)= \gamma\int_{0}^{x}\Pe(t,y) dy.
\end{equation}
It follows from the regularity of $\ue$ that
\begin{equation}
\label{eq:1212}
\lim_{x\to\infty}\left(-\frac{1}{6}\ue^3(t,x)-\eps\px\ue(t,x)\right)=0.
\end{equation}
For \eqref{eq:int-u}, we have that
\begin{equation}
\label{eq:12123}
\lim_{x\to\infty}\int_{0}^{x}\pt\ue(t,y)dy=\int_{0}^{\infty}\pt\ue(t,x) dx = \frac{d}{dt}\int_{0}^{\infty}\ue(t,x)dx=\eps\ptx\Pe(t,0).
\end{equation}
\eqref{eq:int-in-x}, \eqref{eq:1212} and \eqref{eq:12123} give \eqref{eq:lim-di-f}.
\end{proof}

\begin{lemma}\label{lm:P-infty}
Let $T>0$. There exists a function $C(T)>0$, independent on $\eps$, such that
\begin{equation}
\label{eq:l-2-P}
\norm{\Pe}_{L^{\infty}(0,T;L^{2}(0,\infty))}\le C(T).
\end{equation}
In particular, we have that
\begin{align}
\label{eq:L-2-P}
\norm{\Pe(t,\cdot)}_{L^2(0,\infty)}&\le C(T),\\
\label{eq:L-2-px-P-1}
\eps\norm{\px\Pe(t,\cdot)}_{L^2(0,\infty)}&\le C(T),\\
\label{eq:l-2-px-u}
\eps^2 e^{2\gamma t}\int_{0}^{t} e^{-2\gamma s}\left(\ptx\Pe(s,0)+\px\ue(s,0)\right)^2 ds &\le C(T),\\
\label{eq:P-infty-4}
\norm{\Pe}_{L^{\infty}((0,T)\times(0,\infty))}&\le C(T),\\
\label{eq:u-infty-5}
\norm{\ue}_{L^{\infty}((0,T)\times(0,\infty))}&\le C(T).
\end{align}
Moreover, we get
\begin{equation}
\label{eq:pt-px-P}
\eps\left\vert\int_{0}^{t}\!\!\!\int_{0}^{\infty}\Pe\ptx\Pe ds dx\right\vert\le C(T), \quad (t,x)\in (0,T)\times (0,\infty).
\end{equation}
\end{lemma}

\begin{proof}
Let $0<t<T$. We begin by observing that, integrating on $(0,x)$ the second equation of \eqref{eq:SPEsw}, we get
\begin{equation}
\label{eq:P-int-in-0}
\Pe(t,x)=\int_{0}^{x}\ue(t,y)dy + \eps \px\Pe(t,x)- \eps\px\Pe(t,0).
\end{equation}
Differentiating with respect to $t$, we have that
\begin{align*}
\pt\Pe(t,x)&=\frac{d}{dt}\int_{0}^{x}\ue(t,y)dy + \eps \ptx\Pe(t,x)- \eps\ptx\Pe(t,0)\\
&=\int_{0}^{x}\pt\ue(t,x) +\eps \ptx\Pe(t,x)- \eps\ptx\Pe(t,0).
\end{align*}
It follows from \eqref{eq:def-di-f} and \eqref{eq:int-in-x} that
\begin{equation}
\label{eq:equat-per-P}
\begin{split}
\pt\Pe(t,x)=&\gamma\Fe(t,x) +\frac{1}{6}\ue(t,x)^3+\eps\px\ue(t,x)\\
&-\eps\px\ue(t,0)+ \eps \ptx\Pe(t,x)- \eps\ptx\Pe(t,0).
\end{split}
\end{equation}
Multiplying \eqref{eq:equat-per-P} by $\Pe - \eps \px\Pe$, we have that
\begin{equation}
\label{eq:1234}
\begin{split}
(\Pe - \eps \px\Pe)\pt\Pe= &\gamma(\Pe - \eps \px\Pe) \Fe-\frac{1}{6}(\Pe - \eps \px\Pe)\ue^3\\
&-\eps(\Pe - \eps \px\Pe)\px\ue(t,0)+\eps(\Pe -\eps \px\Pe)\px\ue\\
&  +\eps(\Pe -\eps \px\Pe) \ptx\Pe- \eps(\Pe -\eps \px\Pe)\ptx\Pe(t,0).
\end{split}
\end{equation}
Integrating \eqref{eq:1234} on $(0,x)$, for \eqref{eq:SPEsw}, we get
\begin{equation}
\label{eq:1235}
\begin{split}
&\int_{0}^{x}\Pe\pt\Pe dy -\eps \int_{0}^{x}\px\Pe\pt\Pe dy\\
&\quad=\gamma \int_{0}^{x}\Pe\Fe dy - \eps \int_{0}^{x}\Fe\px\Pe dy +\frac{1}{6}\int_{0}^{x}\Pe \ue^3 dy\\
&\qquad -\frac{\eps }{6}\int_{0}^{x}\px\Pe\ue^3dy -\eps \px\ue(t,0)\int_{0}^{y} \Pe dx+\eps^2 \px\ue(t,0)\Pe\\
&\qquad +\eps \int_{0}^{x} \Pe\px\ue dy -\eps^2 \int_{0}^{x}\px\Pe\px\ue dy + \eps \int_{0}^{x}\Pe\ptx\Pe dy\\
&\qquad -\eps^2 \int_{0}^{x}\px\Pe\ptx\Pe dy -\eps \ptx\Pe(t,0)\int_{0}^{x}\Pe dy +\eps^2 \ptx\Pe(t,0)\Pe.
\end{split}
\end{equation}
We observe that, for \eqref{eq:SPEsw},
\begin{equation}
\label{eq:int-by-part}
-\eps \int_{0}^{x}\px\Pe\pt\Pe dy=-\eps \Pe\pt\Pe + \eps \int_{0}^{x}\Pe\ptx\Pe dy.
\end{equation}
Therefore, \eqref{eq:1235} and \eqref{eq:int-by-part} give
\begin{equation}
\begin{split}
\label{eq:P-P-x}
&\int_{0}^{x}\Pe\pt\Pe dy+\eps^2 \int_{0}^{x}\px\Pe\ptx\Pe dy\\
&\quad=\eps \Pe\pt\Pe + \gamma \int_{0}^{x}\Pe\Fe dy- \eps\int_{0}^{x}\Fe\px\Pe dy\\
&\qquad +\frac{1}{6}\int_{0}^{x}\Pe \ue^3 dy -\frac{\eps }{6}\int_{0}^{x}\px\Pe \ue^3dy  -\eps\px\ue(t,0)\int_{0}^{y} \Pe dx\\
&\qquad +\eps^2 \px\ue(t,0)\Pe +\eps \int_{0}^{x} \Pe\px\ue dy -\eps^2 \int_{0}^{x}\px\Pe\px\ue dy\\
&\qquad -\eps \ptx\Pe(t,0)\int_{0}^{x}\Pe dy.
\end{split}
\end{equation}
Since
\begin{align*}
\int_{0}^{\infty}\Pe\pt\Pe dx &=\frac{1}{2}\frac{d}{dt}\int_{0}^{\infty}\Pe^2dx,\\
\eps^2 \int_{0}^{\infty}\ptx\Pe\px\Pe dx &= \frac{\eps^2 }{2}\frac{d}{dt}\int_{0}^{\infty}(\px\Pe)^2dx,
\end{align*}
when $x\to\infty$, for \eqref{eq:P-pxP-intfy} and \eqref{eq:P-P-x}, we have that
\begin{equation}
\label{eq:12312}
\begin{split}
&\frac{1}{2}\frac{d}{dt}\int_{0}^{\infty}\Pe^2dx+\frac{\eps^2 }{2}\frac{d}{dt}\int_{0}^{\infty}(\px\Pe)^2dx\\
&\quad= \gamma \int_{0}^{\infty}\Pe\Fe dx - \eps\gamma \int_{0}^{\infty}\px\Pe\Fe dx +\frac{1}{6}\int_{0}^{\infty} \Pe \ue^3dx\\
&\qquad-\frac{\eps }{6} \int_{0}^{\infty}\px\Pe \ue^3 dx -\eps \px\ue(t,0)\int_{0}^{\infty}\Pe dx \\
&\qquad+\eps \int_{0}^{\infty}\Pe\px\ue dx+\eps^2 \int_{0}^{\infty}\px\Pe\px\ue dx-\eps \ptx\Pe(t,0)\int_{0}^{\infty}\Pe dx.
\end{split}
\end{equation}
Due to \eqref{eq:def-di-f} and \eqref{eq:lim-di-f},
\begin{align*}
2\gamma \int_{0}^{\infty}\Pe\Fe dx&=2\gamma \int_{0}^{\infty}\Fe\px\Fe dx=\gamma (\Fe(t,\infty))^2\\
&= \frac{\eps^2 }{\gamma}\left(\ptx\Pe(t,0)+\px\ue(t,0)\right)^2,
\end{align*}
that is
\begin{equation}
\label{eq:342}
\begin{split}
2\gamma \int_{0}^{\infty}\Pe\Fe dx= &\frac{\eps^2 }{\gamma}(\ptx\Pe(t,0))^2\\
&+ \frac{2\eps^2 }{\gamma}\ptx\Pe(t,0)\px\ue(t,0)+ \frac{\eps^2 }{\gamma}(\px\ue(t,0))^2.
\end{split}
\end{equation}
Again by \eqref{eq:lim-di-f},
\begin{equation}
\label{eq:343}
\begin{split}
-2\eps \px\ue(t,0)\int_{0}^{\infty}\Pe dx&= -2\frac{\eps^2 }{\gamma}(\ptx\Pe(t,0))\px\ue(t,0) -2\frac{\eps^2 }{\gamma}(\px\ue(t,0))^2,\\
- 2\eps \ptx\Pe(t,0)\int_{0}^{\infty}\Pe dx&= -2\frac{\eps^2 }{\gamma}(\ptx\Pe(t,0))^2 - 2\frac{\eps^2 }{\gamma}\ptx\Pe(t,0)\px\ue(t,0).
\end{split}
\end{equation}
Therefore, \eqref{eq:12312}, \eqref{eq:342} and \eqref{eq:343} give
\begin{align*}
&\frac{d}{dt}\left(\int_{0}^{\infty}\Pe^2dx+ \eps^2 \int_{0}^{\infty}(\px\Pe)^2dx\right)\\
&\quad=\frac{\eps^2 }{\gamma}(\ptx\Pe(t,0))^2+ \frac{2\eps^2 }{\gamma}\ptx\Pe(t,0)\px\ue(t,0)+ \frac{\eps^2 }{\gamma}(\px\ue(t,0))^2\\
&\qquad -2\eps\gamma \int_{0}^{\infty}\px\Pe\Fe dx +\frac{1}{3} \int_{0}^{\infty} \Pe \ue^3 dx +\frac{\eps }{3} \int_{0}^{\infty}\px\Pe \ue^3 dx\\
&\qquad -2\frac{\eps^2 }{\gamma}(\ptx\Pe(t,0))\px\ue(t,0) -2\frac{\eps^2 }{\gamma}(\px\ue(t,0))^2 + 2\eps \int_{0}^{\infty}\Pe\px\ue dx\\
&\qquad +2\eps^2 \int_{0}^{\infty}\px\Pe\px\ue dx  -2\frac{\eps^2 }{\gamma}(\ptx\Pe(t,0))^2 - 2\frac{\eps^2 }{\gamma}\ptx\Pe(t,0)\px\ue(t,0),
\end{align*}
that is,
\begin{equation}
\label{eq:345}
\begin{split}
&\frac{d}{dt}\left(\int_{0}^{\infty}\Pe^2dx+ \eps^2 \int_{0}^{\infty}(\px\Pe)^2dx\right)+\frac{\eps^2 }{\gamma}\left(\ptx\Pe(t,0)+\px\ue(t,0)\right)^2\\
&\quad =  -2\eps \gamma\int_{0}^{\infty}\px\Pe\Fe dx +\frac{1}{3} \int_{0}^{\infty} \Pe\ue^3dx -\frac{\eps }{3} \int_{0}^{\infty}\px\Pe \ue^3 dx\\
&\qquad  + 2\eps \int_{0}^{\infty}\Pe\px\ue dx +2\eps^2 \int_{0}^{\infty}\px\Pe\px\ue dx.
\end{split}
\end{equation}
Thanks to \eqref{eq:SPEsw}, \eqref{eq:P-pxP-intfy}, \eqref{eq:def-di-f} and \eqref{eq:lim-di-f},
\begin{equation}
\label{eq:346}
\begin{split}
-2\eps\gamma \int_{0}^{\infty}\px\Pe\Fe dx&=2\eps\gamma \int_{0}^{\infty}\Pe\px\Fe dx\\
 &= 2\eps \gamma\int_{0}^{\infty} \Pe^2 dx\le 2\gamma  \int_{0}^{\infty} \Pe^2 dx,
\end{split}
\end{equation}
while, for \eqref{eq:SPEsw} and \eqref{eq:P-pxP-intfy},
\begin{equation}
\label{eq:347}
 2\eps \int_{0}^{\infty}\Pe\px\ue=-2\eps \int_{0}^{\infty}\ue\px\Pe dx.
\end{equation}
Hence, \eqref{eq:345}, \eqref{eq:346} and \eqref{eq:347} give
\begin{equation}
\label{eq:348}
\begin{split}
&\frac{d}{dt}\left(\norm{\Pe(t,\cdot)}^2_{L^{2}(0,\infty)}+\eps^2 \norm{\px\Pe(t,\cdot)}^2_{L^{2}(0,\infty)}\right)\\
&\qquad +\frac{\eps^2 }{\gamma}\left(\ptx\Pe(t,0)+\px\ue(t,0)\right)^2\\
&\quad \leq 2\gamma \norm{\Pe(t,\cdot)}^2_{L^{2}(0,\infty)} +\frac{1}{3} \int_{0}^{\infty} \Pe \ue^3 dx -\frac{\eps }{3} \int_{0}^{\infty}\px\Pe \ue^3 dx\\
&\qquad  - 2\eps \int_{0}^{\infty}\ue\px\Pe dx +2\eps^2 \int_{0}^{\infty}\px\Pe\px\ue dx.
\end{split}
\end{equation}
Due \eqref{eq:stima-l-2}, \eqref{eq:258} and the Young inequality,
\begin{equation}
\label{eq:Young1}
\begin{split}
&\frac{1}{3} \int_{0}^{\infty} \Pe \ue^3 dx\le\frac{1}{3} \left\vert \int_{0}^{\infty} \Pe \ue^3 dx\right\vert\le\int_{0}^{\infty}\left \vert\frac{\Pe\ue}{3}\right\vert\ue^2 dx\\
&\quad \le \frac{1}{18}\int_{0}^{\infty}\Pe^2\ue^2dx +\frac{1}{2}\int_{0}^{\infty}\ue^4 dx\le  \frac{1}{18}\norm{\Pe(t,\cdot)}^2_{L^{\infty}(0,\infty)}e^{2\gamma t}\norm{u_{0}}^2_{L^2(0,\infty)}\\
& \qquad +\frac{1}{2}\norm{\ue(t,\cdot)}^2_{L^{\infty}(0,\infty)}e^{2\gamma t}\norm{u_{0}}^2_{L^2(0,\infty)}\\
&\quad \le  \frac{1}{18}e^{3\gamma t}\norm{u_{0}}^3_{L^2(0,\infty)}\norm{\Pe(t,\cdot)}_{L^{2}(0,\infty)} +\frac{1}{2}\norm{\ue(t,\cdot)}^2_{L^{\infty}(0,\infty)}e^{2\gamma t}\norm{u_{0}}^2_{L^2(0,\infty)}.
\end{split}
\end{equation}
For \eqref{eq:stima-l-2}, \eqref{eq:10021} and the Young inequality,
\begin{equation}
\label{eq:Young2}
\begin{split}
&-\frac{\eps }{3} \int_{0}^{\infty}\px\Pe \ue^3 dx\le \frac{\eps }{3}\left\vert\int_{0}^{\infty}\px\Pe \ue^3 dx\right\vert\le \frac{\eps}{3}\int_{0}^{\infty}\vert\px\Pe\ue\vert \ue^2 dx\\
&\quad \le \frac{\eps }{6}\int_{0}^{\infty}(\px\Pe)^2\ue^2 dx + \frac{\eps}{6}\int_{0}^{\infty}\ue^4dx\\
&\quad \le \frac{\eps }{6}\norm{\px\Pe(t,\cdot)}^2_{L^{\infty}(\R)}e^{2\gamma t}\norm{u_{0}}^2_{L^2(0,\infty)}+\frac{1}{6}\norm{\ue(t,\cdot)}^2_{L^{\infty}(0,\infty)}e^{2\gamma t}\norm{u_{0}}^2_{L^2(0,\infty)}\\
&\quad\le \frac{1}{6}e^{4\gamma t}\norm{u_{0}}^4_{L^2(0,\infty)} +\frac{1}{6}\norm{\ue(t,\cdot)}^2_{L^{\infty}(0,\infty)}e^{2\gamma t}\norm{u_{0}}^2_{L^2(0,\infty)}.
\end{split}
\end{equation}
It follows from \eqref{eq:stima-l-2}, \eqref{eq:10021} and the Young inequality that
\begin{equation}
\label{eq:Young3}
\begin{split}
&-2\eps \int_{0}^{\infty}\ue\px\Pe dx\le 2\eps \left\vert \int_{0}^{\infty}\ue\px\Pe dx \right\vert\le  \int_{0}^{\infty}\left\vert\frac{\ue}{\sqrt{\gamma}}\right\vert\vert2\eps\sqrt{\gamma}\px\Pe\vert dx\\
&\quad \le \frac{1}{2\gamma}\norm{\ue(t,\cdot)}^2_{L^2(0,\infty)}+2\eps \gamma \norm{\Pe(t,\cdot)}^2_{L^2(0,\infty)}\\
&\quad \le \frac{1}{2\gamma}e^{2\gamma t}\norm{u_{0}}^2_{L^2(0,\infty)} +2\eps^2 \gamma \norm{\Pe(t,\cdot)}^2_{L^2(0,\infty)}.
\end{split}
\end{equation}
Due to \eqref{eq:10021} and the Young inequality,
\begin{equation}
\label{eq:Young4}
\begin{split}
2\eps^2 \int_{0}^{\infty}\vert\px\Pe\vert\vert\px\ue\vert dx &\le \eps^2\norm{\px\Pe(t,\cdot)}^2_{L^{2}(\R)}+\eps^2\norm{\px\ue(t,\cdot)}_{L^{2}(\R)}\\
& \le e^{2\gamma t}\norm{u_{0}}^2_{L^2(0,\infty)} +\eps^2\norm{\px\ue(t,\cdot)}^2_{L^{2}(\R)}.
\end{split}
\end{equation}
\eqref{eq:348}, \eqref{eq:Young1}, \eqref{eq:Young2} and \eqref{eq:Young3} give
\begin{equation*}
\begin{split}
&\frac{d}{dt}G(t) +\frac{\eps^2 }{\gamma}\left(\ptx\Pe(t,0)+\px\ue(t,0)\right)^2\\
&\quad \leq 2\gamma G(t) +\frac{1}{18}e^{3\gamma t}\norm{u_{0}}^3_{L^2(0,\infty)}\norm{\Pe(t,\cdot)}_{L^{2}(0,\infty)}\\
&\qquad + \frac{2}{3}\norm{\ue(t,\cdot)}^2_{L^{\infty}(0,\infty)}e^{2\gamma t}\norm{u_{0}}^2_{L^2(0,\infty)} + \frac{1}{6}e^{4\gamma t}\norm{u_{0}}^4_{L^2(0,\infty)}\\
&\qquad+\frac{1}{2\gamma}e^{2\gamma t}\norm{u_{0}}^2_{L^2(0,\infty)}+ e^{2\gamma t}\norm{u_{0}}^2_{L^2(0,\infty)}+\eps^2\norm{\px\ue(t,\cdot)}^2_{L^{2}(\R)},
\end{split}
\end{equation*}
that is
\begin{equation}
\label{eq:0001}
\begin{split}
&\frac{d}{dt}G(t)-2\gamma G(t) +\frac{\eps^2 }{\gamma}\left(\ptx\Pe(t,0)+\px\ue(t,0)\right)^2\\
&\quad \le \frac{1}{18}e^{3\gamma t}\norm{u_{0}}^3_{L^2(0,\infty)}\norm{\Pe}_{L^{\infty}(0,T;L^{2}(0,\infty))}+\frac{2}{3}\norm{\ue(t,\cdot)}^2_{L^{\infty}(0,\infty)}e^{2\gamma t}\norm{u_{0}}^2_{L^2(0,\infty)}\\
&\qquad + \frac{1}{6}e^{4\gamma t}\norm{u_{0}}^4_{L^2(0,\infty)}+\frac{1}{2\gamma}e^{2\gamma t}\norm{u_{0}}^2_{L^2(0,\infty)} + e^{2\gamma t}\norm{u_{0}}^2_{L^2(0,\infty)}\\
&\qquad +\eps^2\norm{\px\ue(t,\cdot)}^2_{L^{2}(\R)},
\end{split}
\end{equation}
where
\begin{equation}
\label{eq:def-di-G}
G(t)=\norm{\Pe(t,\cdot)}^2_{L^2(0,\infty)} + \eps^2\norm{\px\Pe(t,\cdot)}^2_{L^2(0,\infty)}.
\end{equation}
The Gronwall Lemma, \eqref{eq:u0eps} and \eqref{eq:stima-l-2} give
\begin{equation}
\label{eq:Grom1}
\begin{split}
&G(t)+\frac{\eps^2 e^{2\gamma t} }{\gamma}\int_{0}^{t}\left(\ptx\Pe(s,0)+\px\ue(s,0)\right)^2 ds\\
&\quad \le\norm{P_{0}}^2_{L^2(0,\infty)}e^{2\gamma t} +C_{0}e^{2\gamma t} + C_{0}e^{2\gamma t}\norm{\Pe}_{L^{\infty}(0,T;L^{2}(0,\infty))}\int_{0}^{t}e^{\gamma s}ds\\
&\qquad + C_{0}e^{2\gamma t}\int_{0}^{t}\norm{\ue(s,\cdot)}^2_{L^{\infty}(0,\infty)}ds+C_{0}t + C_{0}e^{2\gamma t}\int_{0}^{t}e^{2\gamma s} ds.
\end{split}
\end{equation}
Due to \eqref{eq:linfty-u} and the Young inequality,
\begin{equation}
\label{eq:Young6}
\begin{split}
\norm{\ue(t,\cdot)}^2_{L^{\infty}(0,\infty)}\le& \norm{u_0}^2_{L^\infty(0,\infty)}+2\gamma\norm{u_0}_{L^\infty(0,\infty)}\int_{0}^{t}\norm{\Pe(s,\cdot)}_{L^{\infty}(0,\infty)}ds\\
&+\gamma^2\left(\int_{0}^{t}\norm{\Pe(s,\cdot)}_{L^{\infty}(0,\infty)}ds\right)^2\\
\le &2\norm{u_0}^2_{L^\infty(0,\infty)}+ \gamma^2\left(\int_{0}^{t}\norm{\Pe(s,\cdot)}_{L^{\infty}(0,\infty)}ds\right)^2.
\end{split}
\end{equation}
It follows from \eqref{eq:258}, \eqref{eq:Young6} and the Jensen inequality that
\begin{equation}
\label{eq:Jensen1}
\begin{split}
\norm{\ue(t,\cdot)}^2_{L^{\infty}(0,\infty)}\le &2\norm{u_0}^2_{L^\infty(0,\infty)} + \gamma^2 t  \int_{0}^{t}\norm{\Pe(s,\cdot)}^2_{L^{\infty}(0,\infty)}ds\\
\le& C_{0} + \gamma C_{0} t \int_{0}^{t} e^{\gamma s}\norm{\Pe(s,\cdot)}_{L^{2}(0,\infty)} ds.
\end{split}
\end{equation}
Therefore
\begin{equation}
\label{eq:u-infty}
\norm{\ue(t,\cdot)}^2_{L^{\infty}(0,\infty)}\le C_{0} + C(T)\norm{\Pe}_{L^{\infty}(0,T;L^{2}(0,\infty))}
\end{equation}
\eqref{eq:def-di-G}, \eqref{eq:Grom1} and \eqref{eq:u-infty} give
\begin{equation}
\label{eq:00134}
\begin{split}
&\norm{\Pe(t,\cdot)}^2_{L^2(0,\infty)} + \eps^2\norm{\px\Pe(t,\cdot)}^2_{L^2(0,\infty)}\\
&\qquad +\frac{\eps^2 e^{2\gamma t}}{\gamma}\int_{0}^{t} e^{-2\gamma s} \left(\ptx\Pe(s,0)+\px\ue(s,0)\right)^2 ds\\
&\quad\le C(T)+C(T)\norm{\Pe}_{L^{\infty}(0,T;L^{2}(0,\infty))}.
\end{split}
\end{equation}
It follows from \eqref{eq:00134} that
\begin{equation*}
\norm{\Pe}^2_{L^{\infty}(0,T;L^{2}(0,\infty))}-C(T)\norm{\Pe}_{L^{\infty}(0,T;L^{2}(0,\infty))}-C(T)\le 0,
\end{equation*}
which gives \eqref{eq:l-2-P}.

\eqref{eq:L-2-P}, \eqref{eq:L-2-px-P-1} and \eqref{eq:l-2-px-u} follow from \eqref{eq:00134} and \eqref{eq:l-2-P}. \eqref{eq:258} and \eqref{eq:L-2-P} give \eqref{eq:P-infty-4}, while \eqref{eq:u-infty-5} follows from \eqref{eq:linfty-u} and \eqref{eq:P-infty-4}.

Let us show that \eqref{eq:pt-px-P} holds true. We begin by observing that, thanks to \eqref{eq:stima-l-2},
\begin{equation}
\begin{split}
\label{eq:ux-12}
\eps\int_{0}^{t}\norm{\px\ue(s,\cdot)}^2_{L^2(0,\infty)}ds&\le \eps e^{2\gamma t} \int_{0}^{t}e^{-2\gamma s}\norm{\px\ue(s,\cdot)}^2_{L^2(0,\infty)}ds\\
&\le \frac{e^{2\gamma t}}{2}\norm{u_{0}}^2_{L^2(0,\infty)}\le C(T).
\end{split}
\end{equation}
Multiplying \eqref{eq:equat-per-P} by $\Pe$, an integration on $(0,\infty)$ gives
\begin{align*}
2\eps\int_{0}^{\infty}\Pe\ptx\Pe dx=& \frac{d}{dt}\norm{\Pe(t,\cdot)}^2_{L^2(0,\infty)}-2\gamma\int_{0}^{\infty}\Pe\Fe dx +2\int_{0}^{\infty}\Pe f(\ue)dx\\
& -2\eps\int_{0}^{\infty}\Pe\px\ue dx + 2\eps \px\ue(t,0)\int_{0}^{\infty}\Pe dx\\
& + 2\eps\ptx\Pe(t,0)\int_{0}^{\infty}\Pe dx.
\end{align*}
It follows from \eqref{eq:def-di-f}, \eqref{eq:lim-di-f}, \eqref{eq:342} and \eqref{eq:343} that
\begin{align*}
2\eps\int_{0}^{\infty}\Pe\ptx\Pe dx=& \frac{d}{dt}\norm{\Pe(t,\cdot)}^2_{L^2(0,\infty)}- \frac{\eps^2}{\gamma}(\ptx\Pe(t,0))^2\\
&- \frac{2\eps^2}{\gamma}\ptx\Pe(t,0)\px\ue(t,0)- \frac{\eps^2}{\gamma}(\px\ue(t,0))^2\\
& +\frac{1}{6}\int_{0}^{\infty}\Pe \ue^3dx -2\eps\int_{0}^{\infty}\Pe\px\ue dx\\
&+2\frac{\eps^2}{\gamma}\ptx\Pe(t,0)\px\ue(t,0) +2\frac{\eps^2}{\gamma}(\px\ue(t,0))^2\\
&+2\frac{\eps^2}{\gamma}(\ptx\Pe(t,0))^2  +2\frac{\eps^2}{\gamma}\ptx\Pe(t,0)\px\ue(t,0),
\end{align*}
that is,
\begin{align*}
2\eps\int_{0}^{\infty}\Pe\ptx\Pe dx=& \frac{d}{dt}\norm{\Pe(t,\cdot)}^2_{L^2(0,\infty)}-\frac{\eps^2}{\gamma}\left(\ptx\Pe(t,0)-\px\ue(t,0)\right)^2\\
&+\frac{1}{6}\int_{0}^{\infty}\Pe \ue^3dx +2\eps\int_{0}^{\infty}\Pe\px\ue dx.
\end{align*}
An integration on $(0,t)$ gives
\begin{align*}
2\eps\int_{0}^{t}\!\!\!\int_{0}^{\infty}\Pe\ptx\Pe dsdx=&\norm{\Pe(t,\cdot)}^2_{L^2(0,\infty)}-\norm{P_{\eps,0}}^2_{L^2(0,\infty)}\\
&-\frac{\eps^2}{\gamma}\int_{0}^{t}\left(\ptx\Pe(s,0)-\px\ue(s,0)\right)^2 ds\\
&+\frac{1}{6}\int_{0}^{t}\!\!\!\int_{0}^{\infty}\Pe \ue^3dx -2\eps\int_{0}^{t}\!\!\!\int_{0}^{\infty}\Pe\px\ue dsdx.
\end{align*}
It follows from \eqref{eq:u0eps}, \eqref{eq:stima-l-2}, \eqref{eq:L-2-P}, \eqref{eq:P-infty-4}, \eqref{eq:u-infty-5} and \eqref{eq:Young1} that
\begin{align*}
2\eps\left\vert\int_{0}^{t}\!\!\!\int_{0}^{\infty}\Pe\ptx\Pe dsdx\right\vert\le&\norm{\Pe(t,\cdot)}^2_{L^2(0,\infty)}+\norm{P_{\eps,0}}^2_{L^2(0,\infty)}\\       &+\frac{\eps^2}{\gamma}\int_{0}^{t}\left(\ptx\Pe(s,0)-\px\ue(s,0)\right)^2 ds\\
&+2\eps\int_{0}^{t}\!\!\!\int_{0}^{\infty}\vert\Pe\vert\vert\px\ue\vert dsdx+C(T)\\
\le&\norm{P_{0}}^2_{L^2(0,\infty)}+\frac{\eps^2 e^{2\gamma t}}{\gamma}\int_{0}^{t}e^{-2\gamma s}\left(\ptx\Pe(s,0)-\px\ue(s,0)\right)^2 ds\\
&+2\eps\int_{0}^{t}\!\!\!\int_{0}^{\infty}\vert\Pe\vert\vert\px\ue\vert dsdx+C(T)\\
\le & \norm{P_{0}}^2_{L^2(0,\infty)} +2\eps\int_{0}^{t}\!\!\!\int_{0}^{\infty}\vert\Pe\vert\vert\px\ue\vert dsdx+C(T)
\end{align*}
Due to \eqref{eq:L-2-P} and the Young inequality,
\begin{equation}
\label{eq:Young7}
\begin{split}
&2\eps\int_{0}^{\infty}\vert\Pe\vert\vert\px\ue\vert dx=2\int_{0}^{\infty}\vert\Pe\vert\vert\eps\px\ue\vert dx\\
&\quad\le \norm{\Pe(t,\cdot)}^2_{L^2(0,\infty)}+\eps^2\norm{\px\ue(t,\cdot)}^2_{L^2(0,\infty)}\\
&\quad\le C(T) + \eps^2\norm{\px\ue(t,\cdot)}^2_{L^2(0,\infty)}.
\end{split}
\end{equation}
Thus, for \eqref{eq:ux-12} and \eqref{eq:Young7}, we have that
\begin{equation*}
2\eps\int_{0}^{t}\!\!\!\int_{0}^{\infty}\vert\Pe\vert\vert\px\ue\vert dsdx \le \int_{0}^{t}\norm{\Pe(s,\cdot)}^2_{L^2(0,\infty)} ds +  \eps^2\int_{0}^{t}\norm{\px\ue(s,\cdot)}^2_{L^2(0,\infty)} ds\le C(T).
\end{equation*}
Therefore,
\begin{equation*}
2\eps\left\vert\int_{0}^{t}\!\!\!\int_{0}^{\infty}\Pe\ptx\Pe dsdx\right\vert\le \norm{P_{0}}^2_{L^2(0,\infty)}+C(T),
\end{equation*}
which gives \eqref{eq:pt-px-P}.
\end{proof}
Let us continue by proving the existence of  a distributional solution
to  \eqref{eq:SPE}, \eqref{eq:boundary}, \eqref{eq:init}  satisfying \eqref{eq:SPEentropyboundary}.
\begin{lemma}\label{lm:conv}
Let $T>0$. There exists a function $u\in L^{\infty}((0,T)\times (0,\infty))$ that is a distributional
solution of \eqref{eq:SPEw} and satisfies \eqref{eq:SPEentropyboundary}  for every convex entropy $\eta\in C^2(\R)$.
\end{lemma}

We  construct a solution by passing
to the limit in a sequence $\Set{u_{\eps}}_{\eps>0}$ of viscosity
approximations \eqref{eq:SPEsw}. We use the
compensated compactness method \cite{TartarI}.

\begin{lemma}\label{lm:conv-u}
Let $T>0$. There exists a subsequence
$\{\uek\}_{k\in\N}$ of $\{\ue\}_{\eps>0}$
and a limit function $  u\in L^{\infty}((0,T)\times(0,\infty))$
such that
\begin{equation}\label{eq:convu}
    \textrm{$\uek \to u$ a.e.~and in $L^{p}_{loc}((0,T)\times(0,\infty))$, $1\le p<\infty$}.
\end{equation}
In particular, \eqref{eq:u-media-nulla} holds true.\\
Moreover, we have
\begin{equation}
\label{eq:conv-P}
\textrm{$\Pek \to P$ a.e.~and in $L^{p}_{loc}(0,T;W^{1,p}_{loc}(0,\infty))$, $1\le p<\infty$},
\end{equation}
where
\begin{equation}
\label{eq:tildeu}
P(t,x)=\int_0^x u(t,y)dy,\qquad t\ge 0,\quad x\ge 0.
\end{equation}
\end{lemma}

\begin{proof}
Let $\eta:\R\to\R$ be any convex $C^2$ entropy function, and let
$q:\R\to\R$ be the corresponding entropy
flux defined by $q'(u)=-\frac{u^2}{2}\eta'(u)$.
By multiplying the first equation in \eqref{eq:SPEsw} with
$\eta'(\ue)$ and using the chain rule, we get
\begin{equation*}
    \pt  \eta(\ue)+\px q(\ue)
    =\underbrace{\eps \pxx \eta(\ue)}_{=:\CLea_{1,\eps}}
    \, \underbrace{-\eps \eta''(\ue)\left(\px  \ue\right)^2}_{=: \CLea_{2,\eps}}
     \, \underbrace{+\gamma\eta'(\ue) \Pe}_{=: \CLea_{3,\eps}},
\end{equation*}
where  $\CLea_{1,\eps}$, $\CLea_{2,\eps}$, $\CLea_{3,\eps}$ are distributions.

Let us show that
\begin{equation*}
\label{eq:H1}
\textrm{$\CLea_{1,\eps}\to 0$ in $H^{-1}((0,T)\times(0,\infty))$, $T>0$.}
\end{equation*}

Since
\begin{equation*}
\eps\pxx\eta(\ue)=\px(\eps\eta'(\ue)\px\ue),
\end{equation*}
for \eqref{eq:stima-l-2} and \eqref{eq:u-infty-5},
\begin{align*}
\norm{\eps\eta'(\ue)\px\ue}^2_{L^2((0,T)\times (0,\infty))}&\le\eps ^2\norm{\eta'}^2_{L^{\infty}(I_T)}\int_{0}^{T}\norm{\px\ue(s,\cdot)}^2_{L^2(0,\infty)}ds\\
&\le\eps\norm{\eta'}^2_{L^{\infty}(I_T)}C(T)\to 0,
\end{align*}
where
\begin{equation*}
I_T=\left(- C(T), C(T)\right).
\end{equation*}
We claim that
\begin{equation*}
\label{eq:L1}
\textrm{$\{\CLea_{2,\eps}\}_{\eps>0}$ is uniformly bounded in $L^1((0,T)\times(0,\infty))$, $T>0$}.
\end{equation*}
Again by \eqref{eq:stima-l-2} and \eqref{eq:u-infty-5},
\begin{align*}
\norm{\eps\eta''(\ue)(\px\ue)^2}_{L^1((0,T)\times (0,\infty))}&\le
\norm{\eta''}_{L^{\infty}(I_T)}\eps
\int_{0}^{T}\norm{\px\ue(s,\cdot)}^2_{L^2(0,\infty)}ds\\
&\le \norm{\eta''}_{L^{\infty}(I_T)}C(T).
\end{align*}
We have that
\begin{equation*}
\textrm{$\{\CL_{3,\eps}\}_{\eps>0}$ is uniformly bounded in $L^1_{loc}((0,T)\times (0,\infty))$, $T>0$.}
\end{equation*}
Let $K$ be a compact subset of $(0,T)\times (0,\infty)$. For \eqref{eq:P-infty-4} and \eqref{eq:u-infty-5},
\begin{align*}
\norm{\gamma\eta'(\ue)\Pe}_{L^1(K)}&=\gamma\int_{K}\vert\eta'(\ue)\vert\vert\Pe\vert
dtdx\\
&\leq \gamma
\norm{\eta'}_{L^{\infty}(I_T)}\norm{\Pe}_{L^{\infty}((0,T)\times\R)}\vert K \vert .
\end{align*}
Therefore, Murat's Lemma \cite{Murat:Hneg} implies that
\begin{equation}
\label{eq:GMC1}
    \text{$\left\{  \pt  \eta(\ue)+\px q(\ue)\right\}_{\eps>0}$
    lies in a compact subset of $\Hneg((0,T)\times(0,\infty))$.}
\end{equation}
\eqref{eq:u-infty-5}, \eqref{eq:GMC1}, and the Tartar's compensated compactness method \cite{TartarI} give the existence of a subsequence
$\{\uek\}_{k\in\N}$ and a limit function $  u\in L^{\infty}((0,T)\times(0,\infty)),\,T>0,$
such that \eqref{eq:convu} holds.

Let us show that \eqref{eq:u-media-nulla} holds true.\\
We begin by proving that
\begin{equation}
\label{eq:px-0}
\textrm{$\eps\px\Pe(\cdot,0)\to 0$ in $L^{\infty}(0,T)$, $T>0$.}
\end{equation}
For \eqref{eq:stima-l-2} and \eqref{eq:10021},
\begin{equation*}
\eps\norm{\px\Pe(\cdot,0)}_{L^{\infty}(0,T)}\leq \sqrt{\eps} e^{\gamma T}\norm{u_0}^2_{L^2(0,\infty)}=\sqrt{\eps}C(T)\to 0,
\end{equation*}
that is \eqref{eq:px-0}.\\
Therefore, \eqref{eq:u-media-nulla} follows from \eqref{eq:int-u}, \eqref{eq:convu} and \eqref{eq:px-0}.

Finally, we prove \eqref{eq:conv-P}.\\
We show that
\begin{equation}
\label{eq:px-1}
\textrm{$\px\Pe\to 0$ in $L^{\infty}((0,T)\times(0,\infty))$, $T>0$.}
\end{equation}
It follows from \eqref{eq:stima-l-2} and \eqref{eq:10021} that
\begin{equation*}
\eps\norm{\px\Pe}_{L^{\infty}((0,T)\times(0,\infty))}\leq \sqrt{\eps} e^{\gamma T}\norm{u_0}^2_{L^2(0,\infty)}=\sqrt{\eps}C(T)\to 0,
\end{equation*}
that is \eqref{eq:px-1}.\\
Then, \eqref{eq:P-int-in-0}, \eqref{eq:convu}, \eqref{eq:px-0}, \eqref{eq:px-1} and  the H\"older inequality give \eqref{eq:conv-P}.

Moreover, \cite[Theorem 1.1]{CKK} tells us that the limit $u$ admits strong boundary trace $u^\tau_0$ at
$(0,\infty)\times \{x=0\}$.
Since, arguing as in \cite[Section 3.1]{CKK} (indeed
our solution is obtained as the vanishing viscosity
limit of \eqref{eq:SPEw}), \cite[Lemma 3.2]{CKK} and
the boundedness of the source term $P$ (cf.~\eqref{eq:SPEsmooth})
imply \eqref{eq:SPEentropyboundary}.
\end{proof}
\begin{proof}[Proof of Theorem \ref{th:main}]
Lemma \eqref{lm:conv-u} gives the existence of entropy  solution $u(t,x)$ of \eqref{eq:SPE-u}, or
equivalently \eqref{eq:SPEw}. Moreover, it proves that \eqref{eq:u-media-nulla} holds true.\\
We observe that, fixed $T>0$,  the solutions of \eqref{eq:SPE-u}, or
equivalently \eqref{eq:SPEw}, are bounded in $(0,T)\times (0,\infty)$. Therefore, using \cite[Theorem $2.1$]{CdK}, or \cite[Theorem $2.2.1$]{dR}, $u$ is unique and \eqref{eq:stability} holds true.
\end{proof}

\section{The Cauchy problem}\label{sec:2}
Let us consider now the Cauchy problem associated to \eqref{eq:SPE}.
Since the arguments are similar to the one of the previous section we simply sketch them, highlighting only the differences
between the two problems.

In this section we augment \eqref{eq:SPE} with the initial datum
\begin{equation}
\label{eq:init-1}
u(0,x)=u_0(x), \qquad x\in\R.
\end{equation}
We assume that
\begin{equation}
\label{eq:assinit-1}
u_0\in L^{\infty}(\R)\cap L^{1}(\R), \quad \int_{\R}u_{0}(x) dx =0.
\end{equation}
On the function
\begin{equation}
\label{eq:def-di-P0-1}
P_{0}(x)=\int_{-\infty}^{x}u_{0}(y)dy,
\end{equation}
we assume that
\begin{equation}
\label{eq:l-2-di-P0-1}
\norm{P_{0}}^2_{L^2(\R)}=\int_{\R}\left(\int_{-\infty}^{x} u_{0}(y)dy\right)^2 dx < \infty.
\end{equation}

We rewrite the Cauchy problem  \eqref{eq:SPE}, \eqref{eq:init-1} in the following way
\begin{equation}
\label{eq:SPEw-u-1}
\begin{cases}
\displaystyle \pt u-\frac{1}{6}\px u^3=\gamma \int^{x}_{0} u(t,y) dy,&\qquad t>0, \ x\in\R,\\
u(0,x)=u_0(x), &\qquad x\in\R,
\end{cases}
\end{equation}
or equivalently
\begin{equation}
\label{eq:SPEw-1}
\begin{cases}
\displaystyle\pt u-\frac{1}{6} \px u^3=\gamma P,&\qquad t>0, \ x\in\R ,\\
\px P=u,&\qquad t>0, \ x\in\R,\\
 P(t,0)=0,& \qquad t>0,\\
u(0,x)=u_0(x), &\qquad x\in\R.
\end{cases}
\end{equation}

Due to the regularizing effect of the $P$ equation in \eqref{eq:SPEw-1} we have that
\begin{equation*}
    u\in L^{\infty}((0,T)\times\R)\Longrightarrow P\in L^{\infty}((0,T);W^{1_,\infty}(\R)), \quad T>0.
\end{equation*}

\begin{definition}
\label{def:sol-1}
We say that $u\in  L^{\infty}((0,T)\times\R),\ T>0$ is an entropy solution of the initial
value problem \eqref{eq:SPE}, and  \eqref{eq:init-1} if
\begin{itemize}
\item[$i$)] $u$ is a distributional solution of \eqref{eq:SPEw-u-1} or equivalently of \eqref{eq:SPEw-1};
\item[$ii$)] for every convex function $\eta\in  C^2(\R)$ the
entropy inequality
\begin{equation}
\label{eq:SPEentropy-1}
\pt \eta(u)+ \px q(u)-\gamma\eta'(u) P\le 0, \qquad     q(u)=-\int^u \frac{\xi^2}{2} \eta'(\xi)\, d\xi,
\end{equation}
holds in the sense of distributions in $(0,\infty)\times\R$.
\end{itemize}
\end{definition}

The main result of this section is the following theorem.

\begin{theorem}
\label{th:main-1}
Assume  \eqref{eq:assinit-1} and \eqref{eq:def-di-P0-1}.
The initial value problem \eqref{eq:SPE}, \eqref{eq:init-1}, possesses an unique entropy solution $u$ in the sense of Definition \ref{def:sol-1}. In particular, we have that
\begin{equation}
\label{eq:u-media-nulla-1}
\int_{\R} u(t,x)dx =0, \quad t>0.
\end{equation}
Moreover, if $u$ and $v$ are two entropy
solutions \eqref{eq:SPE}, \eqref{eq:init-1}, in the sense of Definition \ref{def:sol-1}, the following inequality holds
 \begin{equation}
 \label{eq:stability-1}
\norm{u(t,\cdot)-v(t,\cdot)}_{L^1(-R,R)}\le  e^{C(T) t}\norm{u(0,\cdot)-v(0,\cdot)}_{L^1(-R-C(T)t,R+C(T)t)},
\end{equation}
for almost every $0<t<T$, $R>0$, and some suitable constant $C(T)>0$.
\end{theorem}

A similar result has been proved in \cite{CdK, dR} in the context of locally bounded solutions.

Our existence argument is based on passing to the limit in a vanishing viscosity approximation of \eqref{eq:SPEw-1}.

Fix a small number $0<\eps<1$, and let $\ue=\ue (t,x)$ be the unique classical solution of the following mixed problem \cite{CHK:ParEll}
\begin{equation}
\label{eq:SPEsw-1}
\begin{cases}
\displaystyle\pt \ue -\frac{1}{2}\ue^2\px\ue=\gamma\Pe+ \eps\pxx\ue,&\quad t>0,\ x\in\R,\\
-\eps\pxx\Pe+\px\Pe=\ue,&\quad t>0,\ x\in\R,\\
\Pe(t,0)=0,&\quad t>0,\\
\ue(0,x)=u_{\eps,0}(x),&\quad x\in\R,
\end{cases}
\end{equation}
where $u_{\eps,0}$ is a $C^\infty$ approximation of $u_{0}$ such that
\begin{equation}
\label{eq:u0eps-1}
\begin{split}
&\norm{u_{\eps,0}}_{L^2(\R)}\le \norm{u_0}_{L^2(\R)}, \quad \norm{u_{\eps,0}}_{L^{\infty}(\R)}\le \norm{u_0}_{L^{\infty}(\R)},\\
&\norm{P_{\eps,0}}_{L^2(\R)}\le \norm{P_{0}}_{L^2(\R)},\quad\eps\norm{\px P_{\eps,0}}_{L^2(\R)}\le C_{0},
\end{split}
\end{equation}
and $C_0$ is a constant independent on $\eps$.

Let us prove some a priori estimates on $\ue$ and $\Pe$, denoting with $C_0$ the constants which depend on the initial datum, and $C(T)$ the constants which depend also on $T$.

Arguing as \cite{Cd} and Section \ref{sec:1}, we obtain the following results

\begin{lemma}
\label{lm:cns-1}
For each $t\in (0,\infty)$,
\begin{align}
\label{eq:P-pxP-intfy-1}
\Pe(t,-\infty)=\px\Pe(t,-\infty)=\Pe(t,\infty)=\px\Pe(t,\infty)=&0,\\
\label{eq:equ-L2-stima-1}
\eps^2\norm{\pxx\Pe(t,\cdot)}^2_{L^2(\R)}+ \norm{\px\Pe(t,\cdot)}^2_{L^2(0,\infty)}=&\norm{\ue(t,\cdot)}^2_{L^2(0,\infty)}.
\end{align}
\end{lemma}

\begin{lemma}
\label{lm:2-1}
For each $t\in(0,\infty)$,
\begin{align}
\label{eq:int-u-1}
\int_{\R}\ue(t,x) dx &=0,\\
\label{eq:L-infty-Px-1}
\sqrt{\eps}\norm{\px\Pe(t, \cdot)}_{L^{\infty}(\R)}&\le \norm{u(t,\cdot)}_{L^2(\R)},\\
\label{eq:uP-1}
\int_{\R}\ue(t,x)\Pe(t,x) dx&\le \norm{u(t,\cdot)}^2_{L^2(\R)}.
\end{align}
\end{lemma}

\begin{lemma}
\label{lm:linfty-u-1}
For every $t\in(0,\infty)$,
\begin{equation}
\label{eq:linfty-u-1}
\norm{\ue(t,\cdot)}_{L^\infty(\R)}\le\norm{u_0}_{L^\infty(\R)}+\gamma\int_{0}^{t}\norm{\Pe(s,\cdot)}_{L^{\infty}(\R)}ds.
\end{equation}
\end{lemma}

\begin{lemma}\label{lm:stima-l-2-1}
For each $t\in(0,\infty)$, the inequality holds
\begin{equation}
\label{eq:stima-l-2-1}
\norm{\ue(t,\cdot)}^2_{L^2(\R)}+ 2\eps e^{2\gamma t}\int_{0}^{\infty}e^{-2\gamma s}\norm{\px\ue(s,\cdot)}^2_{L^2(\R)}ds\le e^{2\gamma t}\norm{u_{0}}^2_{L^2(\R)}.
\end{equation}
In particular, we have
\begin{equation}
\label{eq:10021-1}
\begin{split}
\eps\norm{\pxx\Pe(t,\cdot)}_{L^2(\R)}, \norm{\px\Pe(t,\cdot)}_{L^2(\R)}&\le e^{\gamma t}\norm{u_{0}}_{L^2(\R)},\\
\sqrt{\eps}\norm{\px\Pe(t, \cdot)}_{L^{\infty}(\R)}&\le e^{\gamma t}\norm{u_{0}}_{L^2(\R)}.
\end{split}
\end{equation}
Moreover, we get
\begin{align}
\label{eq:258-1}
\norm{\Pe(t,\cdot)}_{L^{\infty}(\R)}&\le \sqrt{2e^{\gamma t}\norm{u_{0}}_{L^2(0,\infty)}\norm{\Pe(t,\cdot)}_{L^2(\R)}},\\
\label{eq:259-1}
\sqrt{\eps}\vert\px\Pe(t,0)\vert &\le e^{\gamma t}\norm{u_{0}}_{L^2(\R)}.
\end{align}
\end{lemma}
\begin{proof}
Arguing as Section \ref{sec:1}, we obtain \eqref{eq:stima-l-2-1}, \eqref{eq:10021-1} and \eqref{eq:258-1}.

Let us show that \eqref{eq:259-1} holds true.
Squaring the equation for $\Pe$ in \eqref{eq:SPEsw-1}, we get
\begin{equation*}
\eps^2(\pxx\Pe)^2+(\px\Pe)^2 - \eps\px((\px\Pe)^2)=\ue^2.
\end{equation*}
An integration on $(-\infty,0)$ and \eqref{eq:P-pxP-intfy-1} give
\begin{equation}
\label{eq:120-7}
\begin{split}
\eps^2\int_{-\infty}^{0}(\pxx\Pe)^2 dx +\int_{-\infty}^{0}(\px\Pe)^2dx+ &\eps(\px\Pe(t,0))^2 \\
&=\int_{-\infty}^{0} \ue^2 dx \le \norm{\ue(t,\cdot)}^2_{L^2(\R)}.
\end{split}
\end{equation}
It follows from \eqref{eq:stima-l-2-1} and \eqref{eq:120-7} that
\begin{equation*}
\eps(\px\Pe(t,0))^2\le e^{2\gamma t}\norm{u_{0}}^2_{L^2(\R)},
\end{equation*}
which gives \eqref{eq:259-1}.
\end{proof}

\begin{lemma}
\label{lm:p8}
For each $t\ge 0$, we have that
\begin{align}
\label{eq:intp-infty}
\int_{0}^{-\infty}\Pe(t,x)dx&=a_{\eps}(t), \\
\label{eq:int+infty}
\int_{0}^{\infty}\Pe(t,x)dx&=a_{\eps}(t),
\end{align}
where
\begin{equation*}
a_{\eps}(t)= \frac{1}{\gamma}\left(\eps\ptx\Pe(t,0)+\frac{1}{6}\ue(t,0)+\eps\px\ue(t,0)\right).
\end{equation*}
Moreover,
\begin{equation}
\label{eq:Pmedianulla}
\int_{\R}\Pe(t,x)dx=0, \quad t\geq 0.
\end{equation}
\end{lemma}
\begin{proof}
We begin by observing that, integrating the second equation of \eqref{eq:SPEsw-1} on $(0,x)$, we have that
\begin{equation}
\label{eq:P-in-0-1}
\int_{0}^{x} \ue(t,y)dy = \Pe(t,x)-\eps\px\Pe(t,x)+\eps\px\Pe(t,0).
\end{equation}
It follows from \eqref{eq:P-pxP-intfy-1} that
\begin{equation}
\label{eq:lim-int-u-1}
\lim_{x\to -\infty}\int_{0}^{x} \ue(t,y)dy=\int_{0}^{-\infty}\ue(t,x)dx =  \eps\px\Pe(t,0).
\end{equation}
Differentiating \eqref{eq:lim-int-u-1} with respect to $t$, we get
\begin{equation}
\label{eq:lim-int-u-in-t}
\frac{d}{dt}\int_{0}^{-\infty}\ue(t,x)dx= \int_{0}^{-\infty}\pt\ue(t,x)dx=\eps\ptx\Pe(t,0).
\end{equation}
Integrating the first equation \eqref{eq:SPEsw-1} on $(0,x)$, we obtain that
\begin{equation}
\begin{split}
\label{eq:int-1-eq-1}
\int_{0}^{x}\pt\ue(t,y) dy &-\frac{1}{6}\ue^3(t,x)+\frac{1}{6}\ue^3(t,0)\\
&-\eps\px\ue(t,x)+ \eps\px\ue(t,0)=\gamma\int_{0}^{x}\Pe(t,y)dy.
\end{split}
\end{equation}
Being $\ue$ a smooth solution of \eqref{eq:SPEsw-1}, we get
\begin{equation}
\label{eq:500-1}
\lim_{x\to-\infty}\Big(-\frac{1}{6}\ue^3(t,x) -\eps\px\ue(t,x)\Big)=0.
\end{equation}
Sending $x\to -\infty$ in \eqref{eq:int-1-eq-1}, for \eqref{eq:lim-int-u-in-t} and \eqref{eq:500-1}, we have
\begin{equation*}
\gamma\int_{0}^{-\infty}\Pe(t,x)dx= \eps\ptx\Pe(t,0) +\frac{1}{6}\ue^3(t,0) + \eps \px\ue(t,0),
\end{equation*}
which gives \eqref{eq:intp-infty}.

Let us show that \eqref{eq:int+infty} holds true. We begin by observing that, for \eqref{eq:P-pxP-intfy-1} and \eqref{eq:P-in-0-1},
\begin{equation*}
\int_{0}^{\infty}\ue(t,x)dx =  \eps\px\Pe(t,0).
\end{equation*}
Therefore,
\begin{equation}
\label{eq:lim-int-u-2}
\lim_{x\to \infty}\int_{0}^{x} \pt\ued(t,y)dy=\int_{0}^{\infty}\pt\ue(t,x)dx =  \eps\ptx\Pe(t,0).
\end{equation}
Again by the regularity of $\ue$,
\begin{equation}
\label{eq:510}
\lim_{x\to\infty}\Big( -\frac{1}{6}\ue^3(t,x)-\eps\px\ue(t,x)\Big)=0.
\end{equation}
It follows from \eqref{eq:int-1-eq-1}, \eqref{eq:lim-int-u-2} and \eqref{eq:510} that
\begin{equation*}
\gamma\int_{0}^{\infty}\Pe(t,x)dx= \eps\ptx\Ped(t,0)+\frac{1}{6}\ue^3(t,0) + \eps\px\ue(t,0),
\end{equation*}
which gives \eqref{eq:int+infty}.

Finally, we prove \eqref{eq:Pmedianulla}. It follows from \eqref{eq:intp-infty} that
\begin{equation*}
\int_{-\infty}^{0}\Pe(t,x)dx = -a_{\eps}(t).
\end{equation*}
Therefore, for \eqref{eq:int+infty},
\begin{align*}
\int_{-\infty}^{0}\Pe(t,x)dx+\int_{0}^{\infty}\Pe(t,x)=\int_{\R} \Pe(t,x)dx =-a_{\eps}(t)+a_{\eps}(t)=0,
\end{align*}
that is \eqref{eq:Pmedianulla}.
\end{proof}
Lemma \ref{lm:p8} says that $\Pe(t,x)$  is integrable at $\pm\infty$.
Therefore, for each $t\ge 0$, we can consider the following function
\begin{equation}
\label{eq:F1}
\Fe(t,x)=\int_{-\infty}^{x}\Pe(t,y)dy.
\end{equation}
\begin{lemma}\label{lm:P-infty-1}
Let $T>0$. There exists a function $C(T)>0$, independent on $\eps$, such that
\begin{equation}
\label{eq:l-2-P-1}
\norm{\Pe}_{L^{\infty}(0,T;L^{2}(\R))}\le C(T).
\end{equation}
In particular, we have that
\begin{align}
\label{eq:L-2-P-1}
\norm{\Pe(t,\cdot)}_{L^2(\R)}&\le C(T),\\
\label{eq:L-2-px-P-1-1}
\eps\norm{\px\Pe(t,\cdot)}_{L^2(\R)}&\le C(T),\\
\label{eq:P-infty-4-1}
\norm{\Pe}_{L^{\infty}((0,T)\times(\R))}&\le C(T),\\
\label{eq:u-infty-5-1}
\norm{\ue}_{L^{\infty}((0,T)\times(\R))}&\le C(T).
\end{align}
Moreover, we get
\begin{equation}
\label{eq:pt-px-P-1}
\eps\left\vert\int_{0}^{t}\!\!\!\int_{\R}\Pe\ptx\Pe ds dx\right\vert\le C(T), \quad t\in (0,T).
\end{equation}
\end{lemma}
\begin{proof}
Integrating the second equation of \eqref{eq:SPEsw-1} on $(-\infty, x)$, for \eqref{eq:P-pxP-intfy-1},  we have that
\begin{equation}
\label{eq:1550-1}
\int_{-\infty}^{x} \ue(t,y)dy=\Pe(t,x) -\eps\px\Pe(t,x).
\end{equation}
Differentiating \eqref{eq:1550-1} with respect to $t$, we get
\begin{equation}
\label{eq:1551-1}
\frac{d}{dt}\int_{-\infty}^{x} \ue(t,y)dy=\int_{-\infty}^{x}\pt \ue(t,y)dy=\pt\Pe(t,x) -\eps\ptx\Pe(t,x).
\end{equation}
It follows from an integration of the first equation of \eqref{eq:SPEsw-1} on $(-\infty, x)$ and \eqref{eq:F1} that
\begin{equation}
\label{eq:1552-1}
\int_{-\infty}^{x}\pt\ue(t,y)dy -\frac{1}{6}\ue^3(t,x) - \eps\px\ue(t,x)=\gamma\Fe(t,x).
\end{equation}
Due to \eqref{eq:1551-1} and \eqref{eq:1552-1}, we have
\begin{equation}
\label{eq:1554-1}
\pt\Pe(t,x)-\eps\ptx\Pe(t,x) =\gamma\Fe(t,x)+\frac{1}{6}\ue^3(t,x) +\eps\px\ue(t,x).
\end{equation}
Multiplying \eqref{eq:1554-1} by $\Pe - \eps\px\Pe$, we have
\begin{equation}
\label{eq:1555-1}
\begin{split}
(\pt\Pe-\eps\ptx\Pe)(\Pe - \eps\px\Pe)= &\gamma\Fe(\Pe - \eps\px\Pe)\\
&+\frac{1}{6}\ue^3 (\Pe - \eps\px\Pe)+\eps\px\ue(\Pe - \eps\px\Pe).
\end{split}
\end{equation}
Integrating \eqref{eq:1555-1} on $(0,x)$, we  have that
\begin{equation}
\label{eq:1222-1}
\begin{split}
\int_{0}^{x}\pt\Pe\Pe dy&-\eps\int_{0}^{x} \pt\Pe\px\Pe dy\\
&-\eps\int_{0}^{x}\Pe \ptx\Pe dy +\eps^2\int_{0}^{x}\ptx\Pe\px\Pe dy\\
=& \gamma\int_{0}^{x}\Fe\Pe dy - \gamma\eps\int_{0}^{x} \Fe\px\Pe dy\\
&+\frac{1}{6}\int_{0}^{x}\ue^3\Pe dy -\frac{1}{6} \eps \int_{0}^{x}\ue^3\px\Pe dy\\
&+\eps\int_{0}^{x}\px\ue\Pe dy - \eps^2\int_{0}^{x}\px\ue\px\Pe dy.
\end{split}
\end{equation}
We observe that, for \eqref{eq:SPEsw-1},
\begin{equation}
\label{eq:int-by-part-1}
-\eps \int_{0}^{x}\px\Pe\pt\Pe dy=-\eps\Pe\pt\Pe + \eps\int_{0}^{x}\Pe\ptx\Pe dy.
\end{equation}
Therefore, \eqref{eq:1222-1} and \eqref{eq:int-by-part-1} give
\begin{equation}
\begin{split}
\label{eq:1223-1}
\int_{0}^{x}\pt\Pe\Pe dy&+ \eps^2\int_{0}^{x}\ptx\Pe\px\Pe dy\\
=& \eps\Pe\pt\Pe + \gamma\int_{0}^{x}\Fe\Pe dy- \gamma\eps\int_{0}^{x} \Fe\px\Pe dy \\
&+\frac{1}{6}\int_{0}^{x}\ue^3\Pe dy -\frac{\eps}{6} \int_{0}^{x}\ue^3\px\Pe dy\\
&+\eps\int_{0}^{x}\px\ue\Pe dy - \eps^2\int_{0}^{x}\px\ue\px\Pe dy.
\end{split}
\end{equation}
Sending $x\to -\infty$, for \eqref{eq:P-pxP-intfy-1}, we get
\begin{equation}
\label{eq:0012-1}
\begin{split}
\int_{0}^{-\infty}\pt\Ped\Ped dy&+ \eps^2\int_{0}^{-\infty}\ptx\Ped\px\Ped dy\\
=& \gamma\int_{0}^{-\infty}\Fe\Pe dy- \gamma\eps\int_{0}^{-\infty} \Fe\px\Pe dy \\
&+\frac{1}{6}\int_{0}^{-\infty}\ue^3\Pe dy - \frac{\eps}{6} \int_{0}^{-\infty}\ue^3\px\Pe dy\\
&+\eps\int_{0}^{-\infty}\px\ue\Pe dy- \eps\int_{0}^{-\infty}\px\ue\px\Pe dy,
\end{split}
\end{equation}
while sending $x\to\infty$,
\begin{equation}
\label{eq:0013-1}
\begin{split}
\int_{0}^{\infty}\pt\Pe\Pe dy&+ \eps^2\int_{0}^{\infty}\ptx\Pe\px\Pe dy\\
=& \gamma\int_{0}^{\infty}\Fe\Pe dy- \gamma\eps\int_{0}^{\infty} \Fe\px\Pe dy \\
&+\frac{1}{6}\int_{0}^{\infty}\ue^3\Pe dy -\frac{\eps}{6} \int_{0}^{\infty}\ue^3\px\Pe dy\\
&+\eps\int_{0}^{\infty}\px\ue\Pe dy - \eps^2\int_{0}^{\infty}\px\ue\px\Pe dy.
\end{split}
\end{equation}
Since
\begin{align*}
\int_{\R}\Pe\pt\Pe dx &=\frac{1}{2}\frac{d}{dt}\int_{\R}\Pe^2dx,\\
\eps^2\int_{\R}\ptx\Pe\px\Pe dx &= \frac{\eps^2}{2}\frac{d}{dt}\int_{\R}(\px\Pe)^2dx,
\end{align*}
it follows from \eqref{eq:0012-1} and \eqref{eq:0013-1} that
\begin{equation}
\label{eq:12312-1}
\begin{split}
\frac{1}{2}\frac{d}{dt}\int_{\R}\Pe^2dx&+\frac{\eps^2}{2}\frac{d}{dt}\int_{\R}(\px\Pe)^2dx\\
=& \gamma\int_{\R}\Fe\Pe dx - \gamma\eps\int_{\R} \Fe\px\Pe dx\\
&+\frac{1}{6}\int_{\R}\ue^3\Pe dx -\frac{\eps}{6} \int_{\R}\ue^3\px\Pe dx\\
&+\eps\int_{\R}\px\ue\Pe dx - \eps^2\int_{\R}\px\ue\px\Pe dx.
\end{split}
\end{equation}
Due to \eqref{eq:Pmedianulla} and \eqref{eq:F1},
\begin{equation}
\label{eq:F-in-infty-1}
\begin{split}
2\gamma\int_{\R}\Fe\Pe dx&=2\gamma\int_{\R}\Fe\px\Fe dx =\gamma(\Fe(t,\infty))^2\\
&=\gamma\left( \int_{\R} \Pe(t,x)dx \right)^2=0.
\end{split}
\end{equation}
It follows from \eqref{eq:12312-1} and \eqref{eq:F-in-infty-1} that
\begin{equation}
\label{eq:12234}
\begin{split}
&\frac{d}{dt}\left(\int_{\R}\Pe^2dx + \eps^2\int_{\R}(\px\Pe)^2dx\right)\\
&\quad=-2\gamma\eps\int_{\R} \Fe\px\Pe dx +\frac{1}{3}\int_{\R}\ue^3\Pe dx-\frac{\eps}{3} \int_{\R}\ue^3\px\Pe dx\\
&\qquad  +2\eps\int_{\R}\px\ue\Pe dx - 2\eps^2\int_{\R}\px\ue\px\Pe dx.
\end{split}
\end{equation}

Due to \eqref{eq:P-pxP-intfy-1}, \eqref{eq:Pmedianulla} and \eqref{eq:F1},
\begin{equation}
\label{eq:346-1}
-2\eps\gamma\int_{\R}\px\Pe\Fe dx=2\eps\gamma\int_{\R}\Pe\px\Fe dx = 2\eps\gamma\int_{\R} \Pe^2 dx\le 2\gamma \int_{\R} \Pe^2 dx,
\end{equation}
while for \eqref{eq:P-pxP-intfy-1},
\begin{equation}
\label{eq:347-1}
\begin{split}
2\eps\int_{\R}\px\ue\Pe dx=&-2\eps\int_{\R}\ue\px\Pe dx.
\end{split}
\end{equation}
Hence, \eqref{eq:346-1} and \eqref{eq:347-1} give
\begin{align*}
&\frac{d}{dt}\left(\norm{\Pe(t,\cdot)}^2_{L^2(\R)} + \eps^2 \norm{\px\Pe(t,\cdot)}^2_{L^2(\R)}\right)\\
&\quad \le  2\gamma \norm{\Pe(t,\cdot)}^2_{L^2(\R)}+\frac{1}{3}\int_{\R}\ue^3\Pe dx-\frac{\eps}{3} \int_{\R}\ue^3\px\Pe dx\\
&\qquad -2\eps\int_{\R}\ue\px\Pe dx - 2\eps^2\int_{\R}\px\ue\px\Pe dx.
\end{align*}
Due to the Young inequality,
\begin{align*}
&\frac{1}{3}\left\vert\int_{\R}\ue^3\Pe dx\right\vert\le \frac{1}{3}\int_{\R}\vert\Pe\vert\vert\ue\vert \ue^2 dx\le \frac{1}{6}\int_{\R}\Pe^2\ue^2dx + \frac{1}{6}\int_{\R}\ue^4 dx\\
&\qquad \le \frac{1}{6}\norm{\Pe(t,\cdot)}^2_{L^{\infty}(\R)}\norm{\ue(t,\cdot)}^2_{L^2(\R)}+\frac{1}{6}\norm{\ue(t,\cdot)}^2_{L^{\infty}(\R)}\norm{\ue(t,\cdot)}^2_{L^2(\R)},\\     &-\frac{\eps}{3} \int_{\R}\ue^3\px\Pe dx\le \frac{\eps}{3}\left\vert\int_{\R}\ue^3\px\Pe dx\right\vert\le \frac{1}{3}\int_{\R}\vert\eps\px\Pe\vert\vert\ue\vert\ue^2 dx\\
&\qquad \le \frac{\eps^2}{6}\int_{\R}(\px\Pe)^2\ue^2 dx + \frac{1}{6}\int_{\R}\ue^4 dx \le \frac{\eps^2}{6}\norm{\px\Pe(t,\cdot)}^2_{L^{\infty}(\R)}\norm{\ue(t,\cdot)}^2_{L^2(\R)}\\
&\qquad\quad +\frac{1}{6}\norm{\ue(t,\cdot)}^2_{L^{\infty}(\R)}\norm{\ue(t,\cdot)}^2_{L^2(\R)},\\
&-2\eps\int_{\R}\ue\px\Pe dx\le 2\eps\left\vert\int_{\R}\ue\px\Pe dx\right\vert\le  \int_{\R}\left\vert\frac{\ue}{\sqrt{\gamma}}\right\vert\vert 2\sqrt{\gamma}\eps\px\Pe\vert dx\\
&\qquad \le \frac{1}{2\gamma}\norm{\ue(t,\cdot)}^2_{L^2(\R)} +2\gamma\eps^2\norm{\px\Pe(t,\cdot)}^2_{L^{2}(\R)}.
\end{align*}
Therefore, we have that
\begin{align*}
&\frac{d}{dt}\left(\norm{\Pe(t,\cdot)}^2_{L^2(\R)} + \eps^2 \norm{\px\Pe(t,\cdot)}^2_{L^2(\R)}\right)\\
&\quad \le  2\gamma \norm{\Pe(t,\cdot)}^2_{L^2(\R)}+2\gamma\eps^2\norm{\px\Pe(t,\cdot)}^2_{L^{2}(\R)}+ \frac{1}{6}\norm{\Pe(t,\cdot)}^2_{L^{\infty}(\R)}\norm{\ue(t,\cdot)}^2_{L^2(\R)} \\
&\qquad+\frac{\eps^2}{6}\norm{\px\Pe(t,\cdot)}^2_{L^{\infty}(\R)}\norm{\ue(t,\cdot)}^2_{L^2(\R)}+\frac{1}{3}\norm{\ue(t,\cdot)}^2_{L^{\infty}(\R)}\norm{\ue(t,\cdot)}^2_{L^2(\R)}\\  &\qquad + \frac{1}{2\gamma}\norm{\ue(t,\cdot)}^2_{L^2(\R)} + 2\eps^2\int_{\R}\vert\px\ue\vert\vert\px\Pe\vert dx.
\end{align*}
Due to The Young inequality,
\begin{align*}
 2\eps^2\int_{\R}\vert\px\ue\vert\vert\px\Pe\vert dx\le& \eps^2 \norm{\px\ue(t,\cdot)}^2_{L^2(\R)}+\eps^2 \norm{\px\Pe(t,\cdot)}^2_{L^2(\R)}\\
\le & \eps^2 \norm{\px\ue(t,\cdot)}^2_{L^2(\R)}+\norm{\px\Pe(t,\cdot)}^2_{L^2(\R)}.
\end{align*}
Hence,
\begin{align*}
\frac{d}{dt}G(t)-2\gamma G(t)\le&\frac{1}{6}\norm{\Pe(t,\cdot)}^2_{L^{\infty}(\R)}\norm{\ue(t,\cdot)}^2_{L^2(\R)}\\ &+\frac{\eps^2}{6}\norm{\px\Pe(t,\cdot)}^2_{L^{\infty}(\R)}\norm{\ue(t,\cdot)}^2_{L^2(\R)}\\
&+\frac{1}{3}\norm{\ue(t,\cdot)}^2_{L^{\infty}(\R)}\norm{\ue(t,\cdot)}^2_{L^2(\R)}\\
&+\frac{1}{2\gamma}\norm{\ue(t,\cdot)}^2_{L^2(\R)}+\eps^2 \norm{\px\ue(t,\cdot)}^2_{L^2(\R)}\\
&+\norm{\px\Pe(t,\cdot)}^2_{L^2(\R)},
\end{align*}
where
\begin{equation}
\label{eq:def-di-G-1}
G(t)=\norm{\Pe(t,\cdot)}^2_{L^2(\R)} + \eps^2\norm{\px\Pe(t,\cdot)}^2_{L^2(\R)}.
\end{equation}
Thanks to \eqref{eq:stima-l-2-1}, \eqref{eq:10021-1} and \eqref{eq:258-1},
\begin{align*}
\frac{1}{6}\norm{\Pe(t,\cdot)}^2_{L^{\infty}(\R)}\norm{\ue(t,\cdot)}^2_{L^2(\R)}\le& \frac{e^{3\gamma t}}{3}\norm{\Pe(t,\cdot)}_{L^2(\R)}\norm{u_{0}}^3_{L^2(\R)},\\
\frac{\eps^2}{6}\norm{\px\Pe(t,\cdot)}^2_{L^{\infty}(\R)}\norm{\ue(t,\cdot)}^2_{L^2(\R)}\le& \frac{e^{4\gamma t}}{6}\norm{u_{0}}^4_{L^2(\R)},\\
\frac{1}{3}\norm{\ue(t,\cdot)}^2_{L^{\infty}(\R)}\norm{\ue(t,\cdot)}^2_{L^2(\R)}\le &\frac{e^{2\gamma t}}{3}\norm{\ue(t,\cdot)}^2_{L^{\infty}(\R)}\norm{u_{0}}^2_{L^2(\R)},\\
\frac{1}{2\gamma}\norm{\ue(t,\cdot)}^2_{L^2(\R)}\le&\frac{e^{2\gamma t}}{2\gamma}\norm{u_{0}}^2_{L^2(\R)},\\
\norm{\px\Pe(t,\cdot)}^2_{L^2(\R)}\le& e^{2\gamma t}\norm{u_{0}}^2_{L^2(\R)}.
\end{align*}
Thus, we get
\begin{align*}
\frac{d}{dt}G(t)-2\gamma G(t)\le& \frac{e^{3\gamma t}}{3}\norm{\Pe}_{L^{\infty}(0,T;L^2(\R))}\norm{u_{0}}^3_{L^2(\R)}+ \frac{e^{4\gamma t}}{6}\norm{u_{0}}^4_{L^2(\R)}\\
&+\frac{e^{2\gamma t}}{3}\norm{\ue(t,\cdot)}^2_{L^{\infty}(\R)}\norm{u_{0}}^2_{L^2(\R)}+e^{2\gamma t}\norm{u_{0}}^2_{L^2(\R)}\\
&+\frac{e^{2\gamma t}}{2\gamma}\norm{u_{0}}^2_{L^2(\R)}+\eps^2 \norm{\px\ue(t,\cdot)}^2_{L^2(\R)}.
\end{align*}
The Gronwall Lemma, \eqref{eq:u0eps-1} and \eqref{eq:def-di-G-1}  give
\begin{equation}
\label{eq:123-1}
\begin{split}
&\norm{\Pe(t,\cdot)}^2_{L^2(\R)} + \eps^2\norm{\px\Pe(t,\cdot)}^2_{L^2(\R)}\\
&\quad\le\norm{P_{0}}^2_{L^2(\R)}e^{2\gamma t} +C_{0}e^{2\gamma t}+ \frac{C_{0}e^{2\gamma t}}{3}\norm{\Pe}_{L^{\infty}(0,T;L^2(\R))}\int_{0}^{t}e^{\gamma s} ds\\
&\qquad +\frac{C_{0}e^{2\gamma t}}{6}\int_{0}^{t}e^{2\gamma s} ds + \frac{C_{0}e^{2\gamma t}}{3}\int_{0}^{t}\norm{\ue(s,\cdot)}^2_{L^{\infty}(\R)}ds\\
&\qquad +C_{0}e^{2\gamma t}(1+t).
\end{split}
\end{equation}
Due to \eqref{eq:linfty-u-1} and the Young inequality,
\begin{equation}
\label{eq:Young6-1}
\begin{split}
\norm{\ue(t,\cdot)}^2_{L^{\infty}(\R)}\le& \norm{u_0}^2_{L^\infty(\R)}+2\gamma\norm{u_0}_{L^\infty(\R)}\int_{0}^{t}\norm{\Pe(s,\cdot)}_{L^{\infty}(\R)}ds\\
&+\gamma^2\left(\int_{0}^{t}\norm{\Pe(s,\cdot)}_{L^{\infty}(\R)}ds\right)^2\\
\le &2\norm{u_0}^2_{L^\infty(\R)}+ \gamma^2\left(\int_{0}^{t}\norm{\Pe(s,\cdot)}_{L^{\R}(\R)}ds\right)^2.
\end{split}
\end{equation}
It follows from \eqref{eq:258-1}, \eqref{eq:Young6-1} and the Jensen inequality that
\begin{equation}
\label{eq:Jensen12}
\begin{split}
\norm{\ue(t,\cdot)}^2_{L^{\infty}(\R)}\le &2\norm{u_0}^2_{L^\infty(\R)} + \gamma^2 t  \int_{0}^{t}\norm{\Pe(s,\cdot)}^2_{L^{\infty}(\R)}ds\\
\le& C_{0} + \gamma C_{0} t \int_{0}^{t} e^{\gamma s}\norm{\Pe(s,\cdot)}_{L^{2}(\R)} ds.
\end{split}
\end{equation}
Therefore,
\begin{equation}
\label{eq:u-infty-1}
\norm{\ue(t,\cdot)}^2_{L^{\infty}(\R)}\le C_{0} + C(T)\norm{\Pe}_{L^{\infty}(0,T;L^{2}(\R))}.
\end{equation}
\eqref{eq:123-1} and \eqref{eq:u-infty-1} give
\begin{equation}
\label{eq:555-1}
\norm{\Pe(t,\cdot)}^2_{L^2(\R)} + \eps^2\norm{\px\Pe(t,\cdot)}^2_{L^2(\R)}\le C(T) + C(T)\norm{\Pe}_{L^{\infty}(0,T;L^{2}(\R))}.
\end{equation}
It follows from \eqref{eq:555-1} that
\begin{equation*}
\norm{\Pe}^2_{L^{\infty}(0,T;L^2(\R))} -C(T)\norm{\Pe}_{L^{\infty}(0,T;L^{2}(\R))}-C(T) \le 0,
\end{equation*}
which gives \eqref{eq:l-2-P-1}.

\eqref{eq:555-1} and \eqref{eq:l-2-P-1} give \eqref{eq:L-2-P-1} and \eqref{eq:L-2-px-P-1-1}. \eqref{eq:258-1} and \eqref{eq:l-2-P-1} give \eqref{eq:P-infty-4-1}, while \eqref{eq:u-infty-5-1} follows from \eqref{eq:linfty-u-1} and \eqref{eq:P-infty-4-1}.

Finally, arguing as Lemma \ref{lm:P-infty}, we obtain \eqref{eq:pt-px-P-1}.
Therefore, the proof is done.
\end{proof}
Let us continue by proving the existence of  a distributional solution
to  \eqref{eq:SPE}, \eqref{eq:init-1}  satisfying \eqref{eq:SPEentropy-1}.
\begin{lemma}\label{lm:conv-1}
Let $T>0$. There exists a function $u\in L^{\infty}((0,T)\times\R)$ that is a distributional
solution of \eqref{eq:SPEw-1} and satisfies  \eqref{eq:SPEentropy-1} for every convex entropy $\eta\in C^2(\R)$.
\end{lemma}
We  construct a solution by passing
to the limit in a sequence $\Set{u_{\eps}}_{\eps>0}$ of viscosity
approximations \eqref{eq:SPEsw-1}. We use the
compensated compactness method \cite{TartarI}.
\begin{lemma}\label{lm:conv-u-1}
Let $T>0$. There exists a subsequence
$\{\uek\}_{k\in\N}$ of $\{\ue\}_{\eps>0}$
and a limit function $  u\in L^{\infty}((0,T)\times\R)$
such that
\begin{equation}\label{eq:convu-1}
    \textrm{$\uek \to u$ a.e.~and in $L^{p}_{loc}((0,T)\times\R)$, $1\le p<\infty$}.
\end{equation}
In particular, \eqref{eq:u-media-nulla-1} holds true.\\
Moreover, we have
\begin{equation}
\label{eq:conv-P-1}
\textrm{$\Pek \to P$ a.e.~and in $L^{p}_{loc}((0,T);W^{1,p}_{loc}(\R))$, $1\le p<\infty$},
\end{equation}
where
\begin{equation}
\label{eq:tildeu-1}
P(t,x)=\int_0^x u(t,y)dy,\qquad t\ge 0,\quad x\in\R.
\end{equation}
\end{lemma}
\begin{proof}
Let $\eta:\R\to\R$ be any convex $C^2$ entropy function, and
$q:\R\to\R$ be the corresponding entropy
flux defined by $q'(u)=-\frac{u^2}{2}\eta'(u)$.
By multiplying the first equation in \eqref{eq:SPEsw-1} with
$\eta'(\ue)$ and using the chain rule, we get
\begin{equation*}
    \pt  \eta(\ue)+\px q(\ue)
    =\underbrace{\eps \pxx \eta(\ue)}_{=:\CL_{1,\eps}}
    \, \underbrace{-\eps \eta''(\ue)\left(\px  \ue\right)^2}_{=: \CL_{2,\eps}}
     \, \underbrace{+\gamma\eta'(\ue) \Pe}_{=: \CL_{3,\eps}},
\end{equation*}
where  $\CL_{1,\eps}$, $\CL_{2,\eps}$, $\CL_{3,\eps}$ are distributions.

Arguing as in Lemma \ref{lm:conv-u}, we have that
\begin{align*}
&\textrm{$\CLea_{1,\eps}\to 0$ in $H^{-1}((0,T)\times\R),\quad T>0$}, \\
&\textrm{$\{\CLea_{2,\eps}\}_{\eps>0}$ is uniformly bounded in $L^1((0,T)\times\R), \quad T>0$},\\
&\textrm{$\{\CLea_{3,\eps}\}_{\eps>0}$ is uniformly bounded in $L^1_{loc}((0,T)\times\R),\quad T>0$}.
\end{align*}
Therefore, Murat's lemma \cite{Murat:Hneg} implies that
\begin{equation}
\label{eq:GMC1-1}
    \text{$\left\{  \pt  \eta(\ue)+\px q(\ue)\right\}_{\eps>0}$
    lies in a compact subset of $\Hneg((0,\infty)\times\R)$.}
\end{equation}
\eqref{eq:u-infty-5-1}, \eqref{eq:GMC1-1} and the
 Tartar's compensated compactness method \cite{TartarI} give the existence of a subsequence
$\{\uek\}_{k\in\N}$ and a limit function $  u\in L^{\infty}((0,T)\times\R)$
such that \eqref{eq:convu-1} holds.

\eqref{eq:u-media-nulla-1} follows from \eqref{eq:int-u-1} and \eqref{eq:convu-1}.

Finally, we prove \eqref{eq:conv-P-1}. We begin by observing that, integrating the second equation of \eqref{eq:SPEsw-1} on $(0,x)$, we have
\begin{equation}
\label{eq:P-u-1}
\Pe(t,x)=\int_{0}^{x}\ue(t,y)dy +\eps\px\Pe(t,x) -\eps\px\Pe(t,0).
\end{equation}
Let us show that
\begin{equation}
\label{eq:px-1-1}
\textrm{$\eps\px\Pe\to 0$ in $L^{\infty}((0,T)\times\R)$, $T>0$.}
\end{equation}
It follows from \eqref{eq:10021-1} that
\begin{equation*}
\eps\norm{\px\Pe}_{L^{\infty}((0,T)\times\R)}\leq \sqrt{\eps} e^{\gamma T}\norm{u_0}^2_{L^2(\R)}=\sqrt{\eps}C(T)\to 0,
\end{equation*}
that is \eqref{eq:px-1-1}.\\
We claim that
\begin{equation}
\label{eq:px-0-1}
\textrm{$\eps\px\Pe(\cdot,0)\to 0$ in $L^{\infty}(0,T)$, $T>0$.}
\end{equation}
Due to \eqref{eq:259-1}, we have that
\begin{equation*}
\eps\norm{\px\Pe(\cdot,0)}_{L^{\infty}(0,T)}\leq \sqrt{\eps} e^{\gamma T}\norm{u_0}^2_{L^2(\R)}=\sqrt{\eps}C(T)\to 0,
\end{equation*}
that is \eqref{eq:px-0-1}.\\
Therefore, \eqref{eq:convu-1}, \eqref{eq:P-u-1}, \eqref{eq:px-1-1}, \eqref{eq:px-0-1} and the H\"older inequality give \eqref{eq:conv-P-1}.
\end{proof}

\begin{proof}[Proof of Theorem \ref{th:main-1}]
Lemma \eqref{lm:conv-u} gives the existence of an entropy solution $u$ of \eqref{eq:SPEw-u-1}, or
equivalently \eqref{eq:SPEw-1}. Moreover, it proves that \eqref{eq:u-media-nulla-1} holds true.

We observe that, fixed $T>0$,  the solutions of \eqref{eq:SPEw-u-1}, or
equivalently \eqref{eq:SPEw-1}, are bounded in $(0,T)\times\R$. Therefore, using \cite[Theorem $3.1$]{CdK}, or \cite[Theorem $2.3.1$]{dR}, $u$ is unique and \eqref{eq:stability-1}
holds true.
\end{proof}

\end{document}